\documentclass{amsart}

\usepackage{amssymb,amsfonts, amsmath, amsthm}
\usepackage[all,arc]{xy}
\usepackage{enumerate}
\usepackage{mathrsfs}
\usepackage{mathtools}
\usepackage[mathscr]{euscript}
\usepackage{array}
\usepackage{appendix}

\usepackage{tikz}
\usepackage{tikz-cd}
\usepackage{mathpazo}
\usepackage{textcomp}
\usepackage{bbold}
\usepackage[bbgreekl]{mathbbol}

\usepackage[pagebackref, colorlinks=true, linkcolor=blue, citecolor = red]{hyperref}
\renewcommand*{\backref}[1]{}
\renewcommand*{\backrefalt}[4]{({%
		\ifcase #1 Not cited.%
		\or On p.~#2%
		\else On pp.~#2%
		\fi%
	})}

\usepackage[nameinlink,capitalise,noabbrev]{cleveref}
\crefname{subsection}{Subsection}{Subsection}

\usetikzlibrary{patterns}

\usepackage{geometry}
\usepackage{todonotes}

\DeclareMathAlphabet{\mathbbe}{U}{bbold}{m}{n}

\def\DDelta{{\mathbbe{\Delta}}}

\newcommand{\n}{\mathscr{N}}
\newcommand{\s}{\mathscr{S}}

\newcommand{\E}{\mathscr{E}}
\newcommand{\cE}{\mathcal{E}}

\newcommand{\C}{\mathscr{C}}
\newcommand{\D}{\mathscr{D}}

\newcommand{\G}{\mathscr{G}}
\newcommand{\R}{\mathscr{R}}
\newcommand{\K}{\mathscr{K}}
\newcommand{\Q}{\mathscr{Q}}
\newcommand{\U}{\mathscr{U}}

\newcommand{\sO}{\mathscr{O}}
\newcommand{\Map}{\mathrm{Map}}
\newcommand{\Hom}{\mathrm{Hom}}

\newcommand{\Sub}{\mathrm{Sub}}
\newcommand{\Eq}{\mathrm{Eq}}
\newcommand{\comma}{,}
\newcommand{\set}{\mathscr{S}\text{et}}
\newcommand{\cat}{\mathscr{C}\text{at}}

\newcommand{\CSS}{\mathscr{C}\mathscr{S}\mathscr{S}}
\newcommand{\Univ}{\mathscr{U}\mathrm{niv}}
\newcommand{\all}{\mathrm{all}}
\newcommand{\Oall}{\sO^{(\all)}}

\newcommand{\id}{\mathrm{id}}
\newcommand{\cO}{\mathcal{O}}
\newcommand{\cM}{\mathcal{M}}

\newcommand{\sset}{\mathrm{s}\mathscr{S}\mathrm{et}}
\newcommand{\morphism}[2]{[#1,#2]}
\newcommand{\ev}{\mathrm{ev}}
\newcommand{\comp}{\mathrm{comp}}
\newcommand{\Fun}{\mathrm{Fun}}
\newcommand{\Pb}{\mathrm{Pb}}
\newcommand{\Fib}{\mathscr{F}\mathrm{ib}}
\newcommand{\Aut}{\mathrm{Aut}}
\newcommand{\Inn}{\mathrm{Inn}}
\newcommand{\Out}{\mathrm{Out}}
\newcommand{\End}{\mathrm{End}}
\newcommand{\colim}{\mathrm{colim}}
\newcommand{\Grph}{\mathscr{G}\mathrm{rph}}

\newcommand{\RFib}{\mathcal{R}\mathcal{F}\mathrm{ib}}
\newcommand{\Cart}{\mathcal{C}\mathrm{art}}
\newcommand{\Comp}{\mathcal{C}\mathrm{omp}}
\newcommand{\Seg}{\mathscr{S}\mathrm{eg}}

\newcommand{\Disc}{\mathcal{D}\mathrm{isc}}
\newcommand{\Cech}{\mathrm{co}\Disc}
\newcommand{\otherComp}{\mathscr{C}\mathrm{omp}}
\newcommand{\SegCat}{\mathscr{S}\mathrm{eg}\mathscr{C}\mathrm{at}}
\newcommand{\FinSet}{\mathscr{F}\mathrm{in}\mathscr{S}\mathrm{et}}
\newcommand{\Grp}{\mathscr{G}\mathrm{rp}}
\newcommand{\Grpd}{\mathscr{G}\mathrm{rpd}}

\newcommand{\ds}{\displaystyle}

\newcommand{\fibref}[1]{\cref{sec:target fib}(\ref{#1})} 

\newcommand{\simpset}[7]{
 \begin{tikzcd}[row sep=0.5in, column sep=0.5in]
   #1 \arrow[r, shorten >=1ex,shorten <=1ex]
   \pgfmatrixnextcell #2 
   \arrow[l, shift left=1.2, "#5"] \arrow[l, shift right=1.2, "#4"'] 
   \arrow[r, shift right, shorten >=1ex,shorten <=1ex ] \arrow[r, shift left, shorten >=1ex,shorten <=1ex] 
   \pgfmatrixnextcell #3 
   \arrow[l] \arrow[l, shift left=2, "#7"] \arrow[l, shift right=2, "#6 "'] 
   \arrow[r, shorten >=1ex,shorten <=1ex] \arrow[r, shift left=2, shorten >=1ex,shorten <=1ex] \arrow[r, shift right=2, shorten >=1ex,shorten <=1ex]
   \pgfmatrixnextcell \cdots 
   \arrow[l, shift right=1] \arrow[l, shift left=1] \arrow[l, shift right=3] \arrow[l, shift left=3] 
 \end{tikzcd}
}

\newtheorem{theone}[equation]{Theorem}
\newtheorem{lemone}[equation]{Lemma}
\newtheorem{propone}[equation]{Proposition}
\newtheorem{corone}[equation]{Corollary}

\theoremstyle{definition}
\newtheorem{defone}[equation]{Definition}
\newtheorem{exone}[equation]{Example}

\theoremstyle{remark}
\newtheorem{remone}[equation]{Remark}

\newtheorem{queone}[equation]{Question}

\numberwithin{equation}{section}

\makeatletter
\def\@seccntformat#1{%
	\expandafter\ifx\csname c@#1\endcsname\c@section\else
	\csname the#1\endcsname\quad
	\fi}
\makeatother

\title{Univalence in Higher Category Theory}

\author{Nima Rasekh}

\date{March 2021}

\address{{\'E}cole Polytechnique F{\'e}d{\'e}rale de Lausanne, SV BMI UPHESS, Station 8, CH-1015 Lausanne, Switzerland}

\email{nima.rasekh@epfl.ch}

\begin{document}

\begin{abstract}
 {\it Univalence} was first defined in the setting of homotopy type theory by Voevodsky \cite{hottbook2013}, who also (along with Kapulkin and Lumsdaine) adapted it to a model categorical setting \cite{kapulkinlumsdaine2012kanunivalent}, which was subsequently generalized to locally Cartesian closed presentable $\infty$-categories by Gepner and Kock \cite{gepnerkock2017univalence}. These definitions were used to characterize various $\infty$-categories as models of type theories. We give a definition for univalent morphisms in finitely complete $\infty$-categories that generalizes the aforementioned definitions and completely focuses on the $\infty$-categorical aspects, characterizing it via representability of certain functors, which should remind the reader of concepts such as adjunctions or limits.
 
 We then prove that in a locally Cartesian closed $\infty$-category (that is not necessarily presentable) univalence of a morphism is equivalent to the completeness of a certain Segal object we construct out of the morphism, characterizing univalence via internal $\infty$-categories, which had been considered in a strict setting by Stenzel \cite{stenzel2019univalence}. We use these results to study the connection between univalence and elementary topos theory. We also study univalent morphisms in the category of groups, the $\infty$-category of $\infty$-categories, and pointed $\infty$-categories.
\end{abstract}

\maketitle
\addtocontents{toc}{\protect\setcounter{tocdepth}{1}}

\tableofcontents

 \noindent

\section{Introduction} \label{sec:introduction}

\subsection{The Rise of Univalence in $\infty$-Category Theory}
The concept of {\it univalence} was first introduced by Voevodsky in the context of type theory with the desire of developing a new foundation of mathematics that permits proof verification of advanced (and homotopical) mathematics \cite{voevodsky2014origins}. This led to a closer connection between homotopy theory and type theory, known as {\it homotopy type theory} \cite{hottbook2013}, also studied by Hofmann and Streicher \cite{hofmannstreicher1998groupoidtypes}, and Awodey and Warren \cite{awodeywarren2009identitytypes}, and related to the effort of studying the connection between type theories and weak categories \cite{gambinogarner2008identitytypewfs,lumsdaine2010weakomegacat,berggarner2011weakgroupoids}. Understanding the role of univalence in type theory is still an important subject of study \cite{ahrenskapulkinshulman2015rezkcompletion,cchm2018cubical,anst2020univalenceshort,anst2021univalencelong}. For a historical account how homotopy theory arose in type theory see \cite{altenkirch2020martin}.

It was Voevodsky himself, along with Kapulkin and Lumsdaine, who gave the first analogue of univalence in a categorical setting. They introduced a notion of univalence of {\it Kan fibrations} and, with some set theoretical assumptions, used that to prove the existence of {\it univalent universes} \cite{kapulkinlumsdaine2012kanunivalent} (also reviewed in \cite{streicher2014hott}). They then use this result to prove Kan complexes give us a {\it model} for homotopy theory and hence deduce that homotopy type theory must be {\it consistent}. One noteworthy aspect of their work is that they use the explicit model structure for Kan complexes, the {\it Kan model structure} on the category of simplicial sets. Indeed, the construction of models of type theories requires certain strict constructions as they want to relate the homotopy theory of Kan complexes to {\it contextual categories} \cite{cartmell1986contexualcategories,streicher1991semanticstypetheory}, which is one way of relating type theories to categories.

This trend of constructing models of {\it intensional type theories} inside strict categories with homotopical data has since been developed extensively, both in model categories \cite{arndtkapulkin2011modelstypetheory,lumsdaineshulman2020goodexcellent} and in fibration categories \cite{kapulkinszumilo2019completequasicat}, where often {\it categories with families} \cite{dybjer1996} and their equivalent {\it natural models} \cite{awodey2018naturalmodels} have been used in order get the desired models.
More specifically, the construction of {\it strict univalent universes} in strict models of homotopy theories, while studied by many others \cite{cisinski2014univalentuniverse,stenzel2019univalence}, has largely been due to Shulman \cite{shulman2015elegantunivalence,shulman2015homotopycanonicity,shulman2017eidiagrams} culminating in \cite{shulman2019inftytoposunivalent}, where he proves that every {\it model topos} \cite{rezk2010toposes} has strict univalent universes, which is expected to completely characterize all combinatorial model categories that permit such a structure.

On the other hand, the study of non-strict univalence in non-strict models of higher categories, so-called {\it $(\infty,1)$-categories}, has been the subject of its own study, primarily due to the work of Gepner and Kock \cite{gepnerkock2017univalence} using the theory of quasi-categories developed in \cite{lurie2009htt,lurie2017ha}. The goal of this work is to generalize their ideas and introduce univalence from an exclusively $\infty$-categorical perspective. In particular, univalence in $\infty$-categories should be thought of as a universal property and so its study should be seen as akin to studying limits, colimits, adjunctions and many other common universal properties.

\subsection{The Internal Object of Equivalences and Internal \texorpdfstring{$\infty$}{oo}-Categories}
 In order to understand univalence in $\infty$-categories, it is instructive to review the construction by Gepner and Kock \cite{gepnerkock2017univalence}. For a given presentable locally Cartesian closed $\infty$-category $\C$ and morphism $p: Y \to X$, using the {\it locally Cartesian closed} structure they construct a space valued presheaf
 $$\Eq(p): (\C_{/X \times X})^{op} \to \s$$
 that takes an object $(f,g):Z \to X \times X$ to the space of equivalences of $f^*Y$ and $g^*Y$ over $Z$, which we denote $\Eq_{/Z}(f^*Y,g^*Y)$. Using the fact that $\C$ is a {\it presentable} $\infty$-category they prove this functor is representable and denote the representing object $\underline{\Eq}_{X \times X}(Y \times X,X \times Y) \to X \times X$. They then observe that there is a natural map of mapping spaces that takes a map $f: Z \to X$ to the identity map in $\Eq_{/Z}(f^*Y,f^*Y)$ and (by the Yoneda lemma) induces a map of representing objects in $\C$
 $$\delta_X:X \to \underline{\Eq}_{X \times X}(Y \times X,X \times Y).$$
 Then they finally give the desired definition and say $p$ is a {\it univalent family} if the map $\delta_X$ is an equivalence. How can we make sense of such a definition?
 
 The underlying idea is that $p: Y \to X$ should be thought of as a universal family or classifying map. Hence a morphism $f: Z \to X$ is expected to classify a certain type of object over $Z$ that is obtained via pulling back $p$ along $f$. Univalence of this morphisms now corresponds to the fact that for two morphism $f,g: Z \to X$ the space of equivalences of $f,g$ inside $\Map(Z,X)$ is equivalent to the space of equivalences of the pulled back objects inside the $\infty$-groupoid $(\C_{/Z})^\simeq$. They then use this description of univalence to completely characterize univalent morphisms in {\it Grothendieck $n$-toposes} where $1 \leq n \leq \infty$.
 
 This approach to univalence is mostly an adoption of the type theoretic approach to the $\infty$-categorical setting, in fact is a direct generalization of the notion of univalence introduced in \cite{kapulkinlumsdaine2012kanunivalent}. Hence, while completely correct, it does not have the appearance of the kind of definitions we would usually encounter in $\infty$-category theory. However, we want to make the point that univalence does fit naturally into the world of higher category and so we will start our discussion with a definition of univalence motivated by \cite{rezk2019leeds} that can be easily described in any (finitely complete) $\infty$-category. Concretely a morphism $p: E \to B$ is univalent if the corresponding natural map of spaces 
 $$p^*(-): \Map_\C(-,B) \to (\C_{/-})^\simeq$$
 induced by pulling back long $p$ is fully faithful. This condition should be thought of as the closest possible condition that an object $B$ can satisfy that want to represent the functor $(\C_{/-})^\simeq$ without running into size issues (as explained in the beginning of \cref{sec:univ target}). Hence, univalence is simply the study of {\it representability} of an important functor, similar to limits, adjunctions, ... .
 
 Starting with this definition we can recover many results in \cite{gepnerkock2017univalence} in any finitely-complete $\infty$-category, which we show in \cref{sec:univ target}. In particular, we show that if the $\infty$-category is also presentable locally Cartesian closed, then this definition will be equivalent to the definition given by Gepner and Kock and hence this can be seen as a direct generalization to finitely complete $\infty$-categories.
 
 The main result of this paper (\cref{the:main theorem}) proves that under mild additional conditions on the $\infty$-category $\C$, the univalence condition on a morphism $p: E \to B$ can be internalized in the following sense: $p$ is univalent in $\C$ if and only if we can construct an {\it internal $\infty$-category object} out of $p$ such that the codomain $B$ is the {\it internal object of equivalences} of this internal $\infty$-category. While more convoluted than the representability condition, the study of internal categories is still a natural part of category theory. It is noteworthy that on the other side, this view on univalence has not yet been developed extensively in type theory as constructing $\infty$-categories in type theory has proven to be quite challenging \cite{finster2018hottest,capriotti2019hottest,finsterreuttervicary2020strictunital}.
 
 While in the setting of finitely complete $\infty$-categories defining internal $\infty$-categories is not too difficult, constructing internal $\infty$-categories is very challenging given all the coherence conditions involved. Hence, we either have to strictify the $\infty$-category in question in order to construct internal categories, which is the approach taken in \cite{stenzel2019univalence}, where Stenzel focuses on univalence in model categories and relative categories \cite{barwickkan2012relativecategory}. The other option is to stay in a non-strict $\infty$-categorical setting, but develop the required $\infty$-categorical machinery, which is our approach. Given the additional complication that results from the approach, we will confine the proof and the necessary theoretical set up to its own section (\cref{sec:proof}), focusing the previous two section on the main theorem and important implications (\cref{sec:univ segal obj}) and motivation from strict category theory (\cref{sec:strict cat}).
 
 Using our understanding of univalence via internal category objects, we will recover many important topos theoretic properties, generalizing results in \cite{gepnerkock2017univalence}. More precisely, we observe how the simple desire to classify all $(-1)$-truncated univalent morphisms in a locally Cartesian closed $\infty$-category naturally leads us to a {\it subobject classifier}, which in the classical setting is the defining property of an elementary topos \cite{tierney1973elementarytopos}. On the other side, the effort to classify $0$-truncated univalent morphisms even in a Grothendieck $1$-topos requires us to study its {\it enveloping Grothendieck $\infty$-topos}. The clear implication is that classifying univalent morphisms pushes us in a higher categorical and topos theoretical direction naturally encouraging us to study elementary versions of higher toposes \cite{rasekh2018elementarytopos}. This is the focus of \cref{sec:univ and topos}.

 The results of \cref{sec:univ and topos} leave us with the question whether univalence is expected to play any role in other $\infty$-categories. In \cref{sec:examples} we will look at two examples: One where univalence is expected to play an important role, but maybe in a different form,  $\cat_\infty$, and one where univalence does not seem to lead to any interesting results, {\it pointed $\infty$-categories}.
 
 \subsection{Strict vs. Non-Strict} 
 Given that we could have chosen the easier route and studied univalence in the strict setting of a model category or relative category, why go through all these complications? One of the main goals is to make univalence more connected to main stream $\infty$-category theory, which has been developed primarily using the model of {\it quasi-categories} \cite{joyal2008notes,joyal2008theory,lurie2009htt,lurie2017ha}, whereas relative categories have not received similar attention. There is also a model-independent approach to $(\infty,1)$-category theory, known as {\it $\infty$-cosmoi} \cite{riehlverity2017inftycosmos,riehlverity2018elements}, however, it only includes non-strict models such as quasi-categories and {\it complete Segal spaces} \cite{rezk2001css}, but not strict models, such as relative categories or {\it simplicial enriched categories} \cite{bergner2007bergnermodelcat}, further providing evidence that univalence should primarily be studied in non-strict model of $(\infty,1)$-categories. Besides the theoretical satisfaction of developing univalence in quasi-categories there are also concrete benefits.
 
 \begin{itemize}
 	\item {\bf Elementary $(\infty,1)$-Topos Theory:} We have already noticed that univalence is related to elementary $(\infty,1)$-topos theory. This has been observed more clearly in \cite{rasekh2018elementarytopos,stenzel2020comprehension}. As topos theory requires extensive use of $\infty$-categorical constructions studying the connection between univalence and topos theory can only be done in the context of quasi-categories, complete Segal spaces or similar $\infty$-cosmoi.
 	\item {\bf Motivic Homotopy Theory:} Motivic homotopy theory is a method for defining homotopy theory in the algebro-geometric setting of schemes, first developed by Morel and Voevodsky \cite{morelvoevodsky1999motivic}. Despite its geometric nature, the resulting $\infty$-category is in fact not a Grothendieck $\infty$-topos \cite[Remark 3.5]{spitzweckostvaer2012motivictwisted} \cite[Proposition 6.7]{gepnerkock2017univalence} and hence the topos theoretic aspects of motivic spaces have been a focus of attention \cite{raptisstrunk2018motivictopos} \cite[Section 6]{gepnerkock2017univalence}. The construction of internal categories presented here could be an important step towards better understanding the topos-theoretic properties of motivic spaces. 
 	\item {\bf $(\infty,2)$-Topos Theory:} The development of {\it Grothendieck $(\infty,1)$-toposes} (called $\infty$-topos in \cite{lurie2009htt}) was strongly guided by a decent understanding of {\it Grothendieck $1$-toposes}. In particular, we know that a locally presentable $1$-category is a Grothendieck $1$-topos if it satisfies certain conditions known as the {\it Giraud axioms} \cite{maclanemoerdijk1994topos} (also called {\it weak descent} in \cite{rezk2010toposes}). This has been lifted directly to the $\infty$-categorical setting by Lurie \cite[Section 6]{lurie2009htt} and Rezk \cite{rezk2010toposes}. Unfortunately the $2$-categorical situation is far less clear. It is not completely understood which conditions we should impose on a locally presentable $1$-category in order to get a working notion of a Grothendieck $2$-topos and it is an active field of research \cite{weber2007twotopos}. Accordingly, there is also no analogous definition for a Grothendieck $(\infty,2)$-topos. 
 	Given this state of affairs, the characterization of univalence via internal categories could be an excellent starting point for future generalizations to $(\infty,2)$-toposes via internal $(\infty,2)$-categories.
 	\item {\bf Other unexpected connections:} There could be other connections that we are right now not aware of. For example, in \cref{cor:univalent complete group} we show that there is a bijection between $1$-truncated connected univalent morphisms over the point and {\it complete groups}, which are currently an object of study in group theory \cite[Section 7]{rotman1995groups} and have not yet been classified. Hence, the study of univalence in $\infty$-categories could help with the classification of objects of interest in other branches of mathematics.
 \end{itemize}

\subsection{Required Background}
 First of all this work is focused on the {\it $\infty$-categorical aspects} of univalence and so no knowledge of (homotopy) type theory is assumed or even necessary. Moreover, a great deal of effort went into concentrating the most technical $\infty$-categorical aspects of the paper in a single section: \cref{sec:proof}. The hope is that the remainder is accessible to anybody who has (or is willing to accept) a general background in higher category theory including some knowledge of limits, colimits and adjunctions. This can be in the context of quasi-categories \cite[Chapter 1]{lurie2009htt}\cite[Sections 1-3]{rezk2017qcats}, complete Segal spaces \cite{rezk2001css}, or $\infty$-cosmoi \cite[Chapters 1,2]{riehlverity2018elements}. 
 
 We also use Cartesian (right) fibrations as a model for higher category (space)-valued presheaves and their Yoneda lemma, again studied extensively for quasi-categories \cite[Chapter 2]{lurie2009htt}, complete Segal spaces \cite[Section 3]{rasekh2017left} and $\infty$-cosmoi \cite[Chapter 5]{riehlverity2018elements}. However, the important results regarding right fibrations will also be reviewed in \cref{sec:target fib}. 
 
\subsection{Connection to other Work}
 Since this work first appeared, many of these results have been also studied by Stenzel, some from a strict angle (using relative categories and model categories) \cite{stenzel2019univalence}, and some from the perspective of the comprehension construction and quasi-categories \cite{stenzel2020comprehension}. Hence, there are references throughout the paper pointing to common results.
 
\subsection{Acknowledgments}
 I want to thank Charles Rezk for many helpful conversations that have motivated this work and the particular definition of univalence chosen here (\cref{def:univalence}).
 I also want to thank Mike Shulman for many helpful conversations on how to relate Segal objects to univalence.
 Moreover, I would like to thank Louis Martini for pointing out an error in a previous version. Finally, I want to thank Philip Hackney for pointing out a mistake in \cref{sec:strict cat} in the previous version (\cref{rem:hackney example}).

\section{From \texorpdfstring{$\infty$}{oo}-Categories to the Target Fibration} \label{sec:target fib}
 As we shall see in the coming sections, the study of univalence in $\infty$-categories is intimately tied to the target fibrations and various derivations thereof. We will hence require various aspects of the theory of fibrations of $\infty$-categories. Most sections only require some basic results that can be found in any source on Cartesian fibrations such as \cite[Section 2.4]{lurie2009htt}, \cite[Chapter 5]{riehlverity2018elements}, \cite[Section 2]{ayalafrancis2020fibrations} and which we will quickly review here. On the other side, \cref{sec:proof} requires a very technical understanding of Cartesian fibrations based on results in \cite{rasekh2017cartesian,rasekh2021cartfibcss,rasekh2021cartfibmarkedvscso}, which will be reviewed at the beginning of that section.
 
 \begin{enumerate}
 	\item \label{item:universes} We fix three universes: {\it small}, {\it large} and {\it very large}.
 	\item \label{item:cosmos} We denote the (very large) $\infty$-cosmos of (large) quasi-categories \cite{joyal2008notes,joyal2008theory} by $\Q\cat$. A {\it theory of $\infty$-categories} is an $\infty$-cosmos $\K$ such that the {\it underlying quasi-category functor} $(-)_0= \Hom_\K(1,-): \K \to \Q\cat$ is a biequivalence of $\infty$-cosmoi with inverse the tensor $1 \otimes -: \Q\cat \to \K$ (also called $\infty$-cosmos of $(\infty,1)$-categories \cite[Definition 1.3.10]{riehlverity2018elements}). According to \cite[Example 1.3.9]{riehlverity2018elements}, examples include the $\infty$-cosmos of {\it quasi-categories} itself, but also the $\infty$-cosmos of {\it complete Segal spaces} $\CSS$ \cite{rezk2001css,joyaltierney2007qcatvssegal}, {\it Segal categories} $\SegCat$ \cite{bergner2007threemodels}, and {\it $1$-complicial sets} $1-\otherComp$ \cite{lurie2009htt}. 
    We define finitely complete (or locally Cartesian closed) $\infty$-categories as defined in \cite{riehlverity2018elements} and note that by \cite[Appendix F]{riehlverity2018elements} in the particular case of quasi-categories the definition coincides with alternative definitions given in \cite{lurie2009htt,gepnerkock2017univalence}.
    \item \label{item:css} If $\K$ is a theory of $\infty$-categories, then we can compose the underlying quasi-category functor with the tensor to get a biequivalence $1 \otimes (-)_0: \CSS \to \K$ that allows us to construct an $\infty$-category out of a complete Segal space. Similarly, there is an inverse $\mathrm{nerve} (\Hom_\K(1,-)):\K \to \CSS$ \cite[Example 1.3.9]{riehlverity2018elements}, that takes an $\infty$-category in $\C$ to its {\it underlying complete Segal space}.
 	\item \label{item:inftycategories} We denote the underlying (very large) quasi-category of $\K$ by $\widehat{\cat_\infty}$ (which is not an object in the $\infty$-cosmos $\Q\cat$) and up to equivalence does not depend on the choice of $\infty$-cosmos. In particular, for two $\infty$-categories $\C,\D$ we denote the Kan complex of functors by $\Map_{\cat_\infty}(\C,\D)$. Moreover, we denote the full subcategory of (very large) $\infty$-groupoids (which we also call spaces) by $\widehat{\s}$. 
 	\item \label{item:underlyinggroupoid} The inclusion functor $\widehat{\s} \hookrightarrow \widehat{\cat_\infty}$ has a right adjoint, which takes an $\infty$-category to its underlying maximal $\infty$-groupoid, which we denote by $(-)^\simeq: \widehat{\cat_\infty} \to \widehat{\s}$.
 	\item \label{item:truncated} Recall a morphism $f:c \to d$ in a finitely complete $\infty$-category $\C$ is {\it $(-2)$-truncated} if it is an equivalence and {\it $n$-truncated} (for $n \geq -1$) if the diagonal $f: c \to c \times_d c$ is $(n-1)$-truncated. 
 	An object is $n$-truncated if the map to the final object is $n$-truncated. In particular, an object is $(-2)$-truncated if and only if it is final. We denote the full sub-$\infty$-category of $n$-truncated objects bt $\tau_n\C$.
 	\item \label{item:cart rfib} 
 	For an $\infty$-category $\C$, a Cartesian fibration over $\C$ models a presheaf from $\C$ valued in $\infty$-categories and a right fibration over $\C$  models a presheaf  from $\C$ valued in spaces. They were first introduced in the context of quasi-categories  \cite[Section 2.4]{lurie2009htt}, but have been generalized to an arbitrary $\infty$-cosmos \cite[Chapter 5]{riehlverity2018elements}, in a way that coincides with the original definition when we restrict to the $\infty$-cosmos of quasi-categories \cite[Appendix F]{riehlverity2018elements}.
 	For a fixed $\infty$-category $\C$, we denote the full sub-quasi-category of $(\widehat{\cat_\infty})_{/\C}$ consisting of Cartesian fibrations by $\Cart_{/\C}$ and the full subcategory of right fibrations  (also called discrete Cartesian fibrations in \cite[Section 5.5]{riehlverity2018elements}) by $\RFib_{/\C}$. Note $\Cart_{/\C}$ (up to categorical equivalence) is independent of the choice of $\infty$-cosmos, as biequivalences of $\infty$-cosmoi preserve and reflect (discrete) Cartesian fibrations \cite[Proposition 10.3.6(x),(xi)]{riehlverity2018elements}.
 	\item \label{item:cart composition} Cartesian (right) fibrations are closed under composition \cite[Lemma 5.2.3]{riehlverity2018elements}.
 	\item \label{item:rfib over space} If $\C$ is an $\infty$-groupoid, then we have an equivalence $\widehat{\s}_{/\C} \simeq \RFib_{/\C}$ (proven for quasi-categories in \cite[Lemma 2.1.3.3]{lurie2009htt}) and similarly $\widehat{\cat_\infty}_{/\C} \simeq \Cart_{/\C}$ (\cite[Theorem 3.1.5.1]{lurie2009htt}).
 	\item \label{item:rep right fib} For every $\infty$-category $\C$ and object $c$, there exists a {\it representable right fibration} $\pi_c:\C_{/c} \to \C$ \cite[Corollary 5.5.13]{riehlverity2018elements}. Moreover, we have the {\it Yoneda lemma}: For a right fibration $\R \to \C$ there is an equivalence of spaces $\Map_{(\cat_\infty)_{/\C}}(\C_{/c},\R) \simeq \R_c$, where $\R_c$ is the fiber of $\R$ over $c$  \cite[Theorem 5.7.1]{riehlverity2018elements}. In particular, taking $\R= \C_{/d}$, we get an equivalence $\Map_{(\cat_\infty)_{/\C}}(\C_{/c},\C_{/d}) \simeq \Map_\C(c,d)$ \cite[Corollary 5.7.16]{riehlverity2018elements}.
	\item \label{item:rep right fib final} For a given right fibration $\R\to \C$, there exists an equivalence of right fibrations $\R\simeq \C_{/c}$ if and only if the $\infty$-category $\R$ has a final object. In particular $\C_{/c}$ has a final object \cite[Corollary 3.5.10]{riehlverity2018elements}.
	\item \label{item:rep right fib limit} We also know that {\it Yoneda preserves limits}. A limit diagram of representable right fibrations is itself representable with representing object given by the limit (\cite[Proposition 4.3.1]{riehlverity2018elements}). 
	\item \label{item:over cat trunc}  We can use over-categories to give alternative characterizations of $(-1)$-truncated morphism. 
	Let $\C$ be an $\infty$-category with finite limits. Then for an object $f:c \to d$ the following are equivalent:
	\begin{enumerate}
		\item[(I)] $f: c \to d$ is $(-1)$-truncated.
		\item[(II)] The induced map (by the Yoneda lemma \fibref{item:rep right fib}) $\C_{/c} \to \C_{/d}$ is fully faithful.
		\item[(III)] The induced map $\C_{/c} \to \C_{/c}\times_{\C_{/d}} \C_{/c}$ is an equivalence (\fibref{item:rep right fib limit}).
		\item[(IV)] For every morphism $g:e \to d$ , the mapping space $\Map_{/d}(e,c)$ is empty or contractible.
		\item[(V)] The commutative square 
		\begin{center}
			\begin{tikzcd}
				c \arrow[r, "f"] \arrow[d, "\Delta"] & d \arrow[d, "\Delta"] \\
				c \times c \arrow[r, "f \times f"] & d \times d
			\end{tikzcd}
		\end{center}
		is a pullback square
	\end{enumerate}
	We call $(-1)$-truncated morphisms also {\it mono}.
	\item \label{item:grothendieck} According to \cite[Theorem 3.2.0.1]{lurie2009htt}, there is an equivalence of quasi-categories 
	\begin{center}
		\begin{tikzcd}
			\Cart_{/\C} \arrow[r, shift left=1.8, "St^+_\C","\bot"'] & \Fun(\C^{op}, \widehat{\cat_\infty}) \arrow[l, shift left=1.8, "Un^+_\C"]
		\end{tikzcd}
	\end{center}
    that restricts to an equivalence \cite[Theorem 2.2.1.2]{lurie2009htt}
    \begin{center}
    	\begin{tikzcd}
    		\RFib_{/\C} \arrow[r, shift left=1.8, "St_\C","\bot"'] & \Fun(\C^{op}, \widehat{\s}) \arrow[l, shift left=1.8, "Un_\C"]
    	\end{tikzcd}
    \end{center}
   takes gives us a correspondence between representable right fibrations $\C_{/c} \to \C$ and representable functors $\Map(-,c):\C^{op} \to \widehat{\s}$.
   \item \label{item:underlyng right fib} We can associate to every Cartesian fibration an {\it underlying right fibration}. It can constructed directly \cite[Corollary 2.4.2.5]{lurie2009htt} or via the functor right fibration that corresponds to the functor 
   $$\C^{op} \to \cat_\infty \xrightarrow{ \ (-)^\simeq \ } \widehat{\s}.$$
\end{enumerate}
 
\section{Univalence via the Target Fibration} \label{sec:univ target}
In this section we define univalent morphisms in finitely complete $\infty$-categories using the target fibration. We then study some important properties and make a first step towards internalizing the definition of univalent morphisms.

Let $\C$ be an $\infty$-category with finite limits and with final object denoted $1_\C$. Following notation in \cite[Section 6.1.1]{lurie2009htt} we denote the $\infty$-category of arrows by $\sO_\C$ and note the target functor $\sO_\C\to \C$ is a Cartesian fibration \cite[Lemma 6.1.1.1]{lurie2009htt}, which corresponds to the functor  
$$\C_{/-}: \C^{op} \to \widehat{\cat_\infty}$$
that takes an object to the over-category, with functoriality given by pullback (\fibref{item:grothendieck}). 

Following notation in \cite[Section 3.3]{gepnerkock2017univalence} we denote the underlying right fibration (\fibref{item:underlyng right fib}) of $\sO_\C$ by $\Oall_\C$ and note it corresponds to the functor $\C^{op} \to \widehat{\s}$ that takes an object $c$ to the underlying space of the over category $(\C_{/c})^\simeq$ (\fibref{item:underlyinggroupoid}). We want to study the representability of the right fibration $\Oall_\C \to \C$ (\fibref{item:rep right fib}). First of all we note that we should not expect $\Oall_\C$ to be commonly representable. Indeed, if it is, say by an object $B$, then we get an equivalence $\Map_\C(1_\C,B) \simeq (\C_{/1_\C})^\simeq \simeq \C^\simeq$, which implies that $\C$ must be a small $\infty$-category. 
 
 Given the lack of final objects, which are precisely the $(-2)$-truncated objects (\fibref{item:truncated}), in $\Oall_\C$ assuming $\C$ is not small, we focus on the next truncation level, which motivates following definition.
 
 \begin{defone} \label{def:univalence}
  A morphism is univalent if it is $(-1)$-truncated (\fibref{item:over cat trunc}) in $\Oall_\C$. 
 \end{defone}
 
 Let's start with an easy example.
  
 \begin{exone} \label{ex:univalent over point}
 	Let $F \to 1_\C$ be a map to the final object. 
 	For a given morphism $Y \to X$, $\Map_{\Oall_\C}(p: Y \to X,F \to 1_\C)$ is non-empty if and only if $ Y \simeq X \times F$ over $X$. Hence, $F \to 1_\C$ is univalent if and only if for every object $X$, $\Map_{\Oall_\C}(\pi_2: Y \times X \to X,F \to 1_\C)$ is contractible.
	This corresponds to the space of equivalences $\Eq_{/X}(F \times X, F \times X)$, which is contractible for all $X$ if and only if $F$ has no self-equivalences. Hence, $F \to 1_\C$ is univalent if and only if the space of self-equivalences of $F$ is contractible. 
 \end{exone}
 
 \begin{lemone} \label{lemma:univ poset}
 	The full subcategory of $\Oall_\C$ consisting of univalent maps is a poset.
 \end{lemone}
 
 \begin{proof}
    By definitions the objects are $(-1)$-truncated and so, by \fibref{item:over cat trunc}, the mapping spaces are empty or contractible.
 \end{proof}

 We denote the full subcategory of univalent maps in $\sO_\C$ by $\Univ_\C$. Notice \cref{lemma:univ poset} was also proven in \cite[Theorem 3.9]{gepnerkock2017univalence} for presentable $\infty$-categories and so this can be seen as a generalization to finitely complete $\infty$-categories.

 \begin{propone} \label{prop:univ final}
 	Let $p:E \to B$ be a morphism in $\C$. Then the following are equivalent.
 	\begin{enumerate}
 		\item $p$ is univalent.
 		\item The unique map $\C_{/B} \to \Oall_\C$ induced by $p$ (via the Yoneda lemma \fibref{item:rep right fib}) is an inclusion.
 		\item The over-category projection $(\Oall_\C)_{/p} \to \Oall_\C$ is an inclusion.
 	\end{enumerate}
 \end{propone}
  
 \begin{proof}
  (1) $\Leftrightarrow$ (3) 
  This is a direct implication of \fibref{item:over cat trunc}.
    
  (2) $\Leftrightarrow$ (3) 
  The map $(\Oall_\C)_{/p} \to \Oall_\C \to \C$ is a right fibration (\fibref{item:cart composition}) with final object (\fibref{item:rep right fib final}) and so, again using \fibref{item:rep right fib final}, is equivalent to $\C_{/B}$.
 \end{proof}
 
  The proposition implies that if $\C$ is a presentable locally Cartesian closed $\infty$-category, then the definition of univalence given here coincides with the definition given by Gepner and Kock \cite[Proposition 3.8]{gepnerkock2017univalence}. We also have following result proven in the context of presentable $\infty$-categories (using different methods) in \cite[Corollary 3.10]{gepnerkock2017univalence} and proven in the context of model categories in \cite[Proposition 5.4]{stenzel2019univalence}.

  \begin{propone} \label{prop:univ if mono}
  	Let $p: E \to B$ be univalent and assume we have following pullback square 
  	\begin{center}
  		\begin{tikzcd}
  			C \arrow[r] \arrow[d, "q"]  & E \arrow[d, "p"] \\
  			D \arrow[r] & B
  		\end{tikzcd}.
  	\end{center}
  	Then $q$ is univalent if and only if $D \to B$ is mono (\fibref{item:over cat trunc}).
  \end{propone}
  
  \begin{proof}
  	If $p$ is univalent then we have an equivalence $\C_{/B} \simeq \Oall_{/p}$ (\cref{prop:univ final}), which restricts to an equivalence of the $(-1)$-truncated objects. 
  	But $(-1)$-truncated objects in $\C_{/B}$ are the $(-1)$-truncated morphisms into $B$ (\fibref{item:truncated}), and the $(-1)$-truncated objects in $\Oall_{/p}$ are the univalent maps that are a pullback of $p$.
  \end{proof}
  
  \begin{remone}
  	If $\C$ is presentable locally Cartesian closed $\infty$-category then we can define bounded local classes of maps \cite[3.3]{gepnerkock2017univalence}. Intuitively, it is a subclass of the morphisms in $\C$ that are closed under base change, satisfies a certain locality condition (sheaf condition) and has some cardinality bound. Those bounded local classes form a poset under inclusion. It is proven in  	\cite[Theorem 3.9]{gepnerkock2017univalence} that in a presentable locally Cartesian closed $\infty$-category there is an equivalence between the poset of bounded local classes and the poset of univalent families. 
  	
  	In an arbitrary finitely complete $\infty$-category we cannot define bounded local classes as there is no internal notion of cardinality, but we can still define univalent morphisms. Thus in the non-presentable setting univalence substitutes the notion of bounded local classes. In other words, we can use univalence to characterize bounded local classes.
  \end{remone}
 
 Let us also make some observations about the poset $\Univ$. For a given univalent morphism $p$ denote the sub-poset of $\Univ$ less than $p$ by $\Univ_{\leq p}$.
 
 \begin{lemone} \label{lemma:univ leq p}
 	For a given univalent morphism $p$, $\Univ_{\leq p}$ has joins.
 \end{lemone}
 
 \begin{proof}
 	By \cref{prop:univ final}, we have an equivalence $\Univ_{\leq B} \simeq \tau_{-1}(\C_{/B})$, where $\tau_{-1}$ represents the full subcategory of $(-1)$-truncated objects.
 	Now, by assumption $\C_{/B}$ has finite limits and so $\tau_{-1}(\C_{/B})$ has finite products as well, as $(-1)$-truncated objects are closed under limits. 
 \end{proof}
  
  \begin{lemone} \label{lemma:univ weak meets}
   Assume that $\Univ$ has weak meets. Then $\Univ$ has joins.
  \end{lemone}
 
 \begin{proof}
  Let $p,q$ be two objects in $\Univ$. By assumption there exists an object $r$ such that $p \leq r, q \leq r$. Hence $p, q $ are objects in $\Univ_{\leq r}$. The result now follows from \cref{lemma:univ leq p}.
 \end{proof}	

 Notice the some assumption along the lines given in \cref{lemma:univ leq p} or \cref{lemma:univ weak meets} are in fact necessary. In fact we will give a counter example for the general case in \cref{lemma:univ groups}.
 
 We move on to preservation of univalence via functors. 

 \begin{propone} \label{prop:over cat univalent}
 	Let $F:\C\to \D$ be a pullback preserving fully faithful functor. Then a morphism $f$ in $\C$ is univalent if $f$ is univalent in $\D$.
 \end{propone}

 \begin{proof}
 	The functor $F$ induced a fully faithful functor $\sO_F: \sO_\C \to \sO_\D$, which restricts to a fully faithful functor $\Oall_F: \Oall_\C \to \Oall_\D$, as $F$ preserves pullbacks. 
 \end{proof}

 In certain circumstances we can improve this result.
 
 \begin{propone}\label{prop:yoneda reflects univalence}
 	The Yoneda embedding reflects univalence.
 \end{propone}
 \begin{proof}
 	One side is \cref{prop:over cat univalent}. For the other side, assume $\Map(-,p):\Map_\C(-,E) \to \Map_\C(-,B)$ is univalent in $\Fun(\C^{op},\s)$. Let $P \to Q$ be an arbitrary natural transformation in $\Fun(\C^{op},\s)$. Then there is a colimit diagram $D: I \to \sO_\C \to \sO_{\Fun(\C^{op},\s)}$ such that $\colim_I D \simeq P \to Q$.  Now by descent \cite[Theorem 6.1.3.9]{lurie2009htt} we have an equivalence $$\Map_{\Oall_{\Fun(\C^{op},\s)}}(\colim_I D,\Map(-,p)) \simeq \lim_I(\Map_{\Oall_{\Fun(\C^{op},\s)}}(\Map(-,q_i),\Map(-,p))) \simeq \lim_I(\Map_{\Oall_\C}(q_i,p)).$$
 	This limit of these spaces is empty or contractible as each $\Map_{\Oall_\C}(q_i,p)$ are empty or contractible by assumption.
 \end{proof}
 
 Assume, $\D \hookrightarrow \C$ is a full subcategory, with localization functor $L: \C\to \D$.
 We say $L: \C \to \D$ is a {\it locally Cartesian closed localization} if for every span $Z \rightarrow X \leftarrow Y$ in $\C$ with $Y$ and $X$ in $\D$ the natural map  
 $$L(Z \times_X Y) \to LZ \times_X Y$$
 is an equivalence. We now have following result with regard to locally Cartesian closed localizations.
 
 \begin{propone} \label{prop:lccc localization reflect univalence}
 	Locally Cartesian closed localizations reflect univalence.
 \end{propone}
 \begin{proof}
 	We can use the same proof as in \cite[Theorem 3.12]{gepnerkock2017univalence}, noticing that \cite[Lemma 1.7]{gepnerkock2017univalence} does not rely on presentability.
 \end{proof}

 Based on the definition we have presented up to here, whenever we want to confirm a morphism is univalent in the $\infty$-category $\C$, we have to study the $\infty$-category $\sO_\C$ instead. We would like to find a way to {\it internalize} the definition: find a criterion internal to $\C$ that determines whether a morphism is univalent. 
 
 Based on \cref{prop:univ final}, a morphism $p: E \to B$ is univalent if and only if the morphism of right fibration $\C_{/B} \to \Oall_\C$ is fully faithful, which is equivalent to $\C_{/B} \to \C_{/B} \times_{\Oall} \C_{/B}$ being an equivalence (\fibref{item:over cat trunc}). If the right fibration $\C_{/B} \times_{\Oall} \C_{/B} \to \C$ is representable, then, by the Yoneda lemma (\fibref{item:rep right fib}), this functor of right fibrations would correspond to a morphism in the category $\C$. 
 Hence, we want to prove that under certain conditions on $\C$, the right fibration $\C_{/B} \times_{\Oall} \C_{/B}$ is representable. 
 
 \begin{lemone} \label{lemma:q}
 		Let $\C$ be a finitely complete $\infty$-category and $p: E \to B$ be a morphism such that pulling back along $E \times B \to B \times B$ has a right adjoint. Then the right fibration $\C_{/B} \times_{\Oall} \C_{/B} \to \C$ is representable.
 \end{lemone}

 \begin{proof}
 	The right fibration $\C_{/B} \times_{\Oall} \C_{/B} \to \C$ factors as follows $\C_{/B} \times_{\Oall} \C_{/B} \to \C_{/B} \times_\C \C_{/B} \to \C$. Hence, it suffices to prove the right fibration $\C_{/B} \times_{\Oall} \C_{/B} \to \C_{/B} \times_\C \C_{/B}$ is representable. We know that $\C_{/B} \times_\C \C_{/B} \simeq \C_{/B\times B}$ (\fibref{item:rep right fib limit}). Now for a given object $(f,g): X \to B \times B$ in $\C_{/B\times B}$, its fiber of $\C_{/B} \times_{\Oall} \C_{/B} \to \C_{/B} \times_\C \C_{/B}$ is given by the space of equivalences $f^*E \simeq g^*E$ over $X$, which we denote by $\Eq_{/X}(f^*E,g^*E)$. This means the right fibration $\C_{/B} \times_{\Oall} \C_{/B} \to \C_{/B \times B}$ corresponds to the functor $\Eq_{/(-)}((-) \times_B E,(-) \times_B E)$ with functoriality given by pullback (\fibref{item:grothendieck}).
 	
 	We want to construct an object that represents this functor of equivalences. This can be done in various ways, one of which is due to Vergura \cite[Lemma 2.8]{vergura2019localization}, who proves that the object $\cE$ given by the following pullback diagram
 	\begin{equation} \label{eq:q}
 		\begin{tikzcd}[row sep=0.5in]
 			\cE \arrow[r] \arrow[d] & \morphism{\id_B \times p}{p \times \id_B}_{B \times B} \underset{B \times B}{\times} \morphism{p \times \id_B}{\id_B \times p}_{B \times B} \underset{B \times B}{\times} \morphism{\id_B \times p}{p \times \id_B}_{B \times B}  \arrow[d, "(c \circ\pi_{12}\comma c' \circ \pi_{23})"] \\
 			B \times B \arrow[r] & \morphism{p \times \id_B}{p \times \id_B}_{B \times B} \underset{B \times B}{\times} \morphism{\id_B \times p}{\id_B \times p}_{B \times B}
 		\end{tikzcd}
 	\end{equation}
 	has the desired property. Here $\pi_{12}$ ($\pi_{23}$) is given by projection on the first two components (second two component) 
 	and $c,c'$ are the appropriate composition maps (which are described in more detail in \cite[Lemma 2.8]{vergura2019localization}).
    Notice the lemma states that $\C$ is a presentable $\infty$-category, however, this assumption is never used in the proof.
\end{proof}    

 Combining this lemma with the explanation before it we now get following result.
  
 \begin{corone} \label{cor:q}
 	Let $\C$ be a finitely complete $\infty$-category and $p: E \to B$ be a morphism such that pulling back along $E \times B \to B \times B$ has a right adjoint. Then $p$ is univalent if and only if the map $B \to \cE$ induced via the Yoneda lemma (\fibref{item:rep right fib}) from the map of right fibrations $$\C_{/B} \to \C_{/\cE} \simeq \C_{/B} \times_{\Oall} \C_{/B}$$ is an equivalence.
 \end{corone}

\section{Univalence in Strict Categories: An Interlude} \label{sec:strict cat}
  The statement of \cref{lemma:q} and its implication, \cref{cor:q}, while correct, do not help us understand the notion of univalence internally. Right now the map constructed in \cref{lemma:q} has only an external meaning, in the sense that it represents the functor with value $\Eq_{/X}(f^*E,g^*E)$. 
  We would like to have a characterization that is internal to $\C$ similar to the characterization of univalence in $\Oall_\C$ (\cref{def:univalence}). 
  
 The lemma is already pointing in a certain direction: The functor $\Eq_{/X}(f^*E,g^*E)$ is naturally a sub-functor of the functor with value $\Map_{/X}(f^*E,g^*E)$. This suggests the idea that we are looking at an internal category, where the objects are given by maps $f: X \to B$ and morphism from $f \to g$ is given by maps $f^*E \to g^*E$ over $X$. We want to make this intuition into a precise definition, however this requires us to understand internal $\infty$-category objects. 

 Before we move on to the general case, we want to focus on case of $1$-categories where we benefit from additional strictness.
 Hence, let us assume $\C$ is a locally Cartesian closed $1$-category and recall a {\it category object} in $\C$ is the data of two objects $\cO$ (classifying objects) and $\cM$ (classifying morphisms), a factorization of the diagonal 
 $$\cO \xrightarrow{ \ \id \ } \cM \xrightarrow{ (s,t) } \cO \times \cO$$ 
 (identity, source and target), and $\mu: \cM \times_{\cO} \cM \to \cM$ (composition) such that certain diagrams (corresponding to the unitality and associativity of categories) commute \cite{eckmannhilton1961internalcat,ehresmann1963internalcat,roberts2012internalcat}.
 
 We want to construct a category object out of $p: E\to B$, motivated by the conditions outlined in the beginning of this section. 
 The description there already suggest that the objects should correspond to $B$. We want to construct the object of morphisms $\cM \to B \times B$, which should have the universal property that for two morphism $f,g: X \to B$, we have an isomorphism $(f,g)^*\cM \simeq \morphism{f^*E}{g^*E}_X$. We can use this universal property to find $\cM$. If $f=\pi_1: B \times B \to B$ and $g=\pi_2: B \times B \to B$, then $(\pi_1,\pi_2)= \id: B \times B \to B \times B$. Hence
 \begin{equation}\label{eq:m}
 	\cM = \morphism{(\pi_1)^*E}{(\pi_2)^*E}_{B \times B}.
 \end{equation} 
 We now want to observe that ($B$, $\cM$) is indeed a category object in $\C$. By construction there is a map $(s,t): \cM \to B \times B$. The map $B \to \cM$ is given by the identity map $\id_E$ in $\morphism{E}{E}_B$. We also want to construct a composition map $\mu: \cM \times_B \cM \to \cM$. We compute
 \begin{center}
 	\begin{tikzcd}
 		((E \times B) \times_{B\times B} \cM) \times_B \cM \arrow[r] \arrow[d] & (E \times B) \times_B (B \times B) \arrow[r] \arrow[d] & E \times B \arrow[r] \arrow[d] & E \arrow[d] \\
 		\cM \times_B \cM \arrow[r, "(s\comma t \comma s \comma t)"] \arrow[rrr, "s\pi_1"', bend right=10] & (B \times B) \times_B (B \times B) \arrow[r, "\pi_1"] & B \times B \arrow[r, "\pi_1"] & B
 	\end{tikzcd}.
 \end{center} 
 meaning $(s\pi_1)^*(\cM \times_B\cM) \simeq ((E \times B) \times_{B\times B} \cM) \times_B \cM$ and similarly we have 
 $(t\pi_2)^*(\cM \times_B\cM) \simeq \cM \times_B (\cM\times_{B \times B} (B \times E))$ and so we need to construct a morphism from the former to the latter. 
 
 Let $\ev: \cM \times_{B \times B} (E \times B) \to B \times E$ be the evaluation map over $B \times B$ (the counit map of the adjunction). Then the desired map 
 $$\comp: (s\pi_1)^*(\cM \times_B\cM) \to (t\pi_2)^*(\cM \times_B\cM)$$ 
 over $B \times B$ is given by the following compositions: 
 \begin{align*}
 ((E \times B) \underset{B\times B}{\times} \cM) \underset{B}{\times} \cM & \xrightarrow{ \ (\pi_2,\ev)\underset{B}{\times} \cM \ } \cM \underset{B\times B}{\times} (B \times E) \underset{B}{\times} \cM & \text{first evaluation} \\
 & \xrightarrow[ \hspace{0.7in} ]{ \cong } \cM \underset{B\times B}{\times} (B \times B) \underset{B}{\times} ((E \times B) \underset{B\times B}{\times} \cM) & \text{as } E \times B \cong E \underset{B}{\times} (B \times B)\\
 & \xrightarrow[ \hspace{0.7in} ]{ \cong } \cM \underset{B}{\times} (E \times B) \underset{B\times B}{\times} \cM & \text{as } \cM \underset{B\times B}{\times} B \times B \cong \cM\\
 & \xrightarrow{ \ \cM \underset{B}{\times} (\pi_2,\ev) \ }  \cM \underset{B}{\times} \cM \underset{B\times B}{\times} B \times E & \text{ second evaluation}
 \end{align*}	 
  Finally, we need to confirm that $\comp$ is unital and associative. However, as it is just defined by applying $(\pi_2,\ev)$ (with some isomorphisms in between) these properties follow straightforwardly from the fact that the evaluation map $\ev$ satisfies them. 
  Hence, we have constructed a category object $\C(p)=(B,\cM,s,t,\comp, \id)$ internal to the $1$-category $\C$. 
 
  How can we characterize the (internal) isomorphisms in the internal category object $(B,\cM)$? For two objects $f,g: X \to B$, a morphism $\alpha: X \to \cM$ from $f$ to $g$ is an isomorphism if there exists morphisms $\beta, \gamma: X \to \cM$ from $g$ to $f$ such that $\comp(\beta,\alpha) = \id(f)$ and $\comp(\alpha,\gamma) = \id(g)$. We can use this explanation to explicitly describe the object of isomorphisms as the pullback of the following span
  \begin{equation} \label{eq:span}
  		\begin{tikzcd}
  		 \cM \times_{B\times B} \cM \times_{B\times B} \cM \arrow[r, "(\comp \circ \pi_{12} \comma  \comp\circ \pi_{23} )"] &[0.9in] \cM\times_{B \times B} \cM & B \times B \arrow[l, "\id \times \id"'] 
  	\end{tikzcd}.
  \end{equation}  
  Comparing this span with the diagram in \ref{eq:q} and the definition of $\cM=\morphism{E \times B}{B \times E}_{B \times B}$, we realize that univalence corresponds to $B$ being the pullback of this span. Internally, this translates to saying that every isomorphism is the identity, hence we have following result.
 
 \begin{corone} \label{cor:univalence int cat}
 	Let $p: E \to B$ be a morphism in a locally Cartesian closed $1$-category $\C$. Then $p$ is univalent if and only if the category object $\C(p)=(B,\cM= \morphism{\pi_1^*E}{\pi_2^*E}_{B \times B},s,t,\comp, \id)$ has no non-trivial isomorphisms.
 \end{corone}
 
 \begin{remone} \label{rem:simp diagram}
 	A reader familiar with category objects might have expected the object of isomorphisms to be characterized via a pullback diagram of the following form 
 	\begin{center}
 	\begin{tikzcd}
 		\cM \times_B \cM \times_B \cM \arrow[r, "(\comp \circ \pi_{12} \comma  \comp\circ \pi_{23} )"] &[0.9in] \cM \times \cM & B \times B \arrow[l, "\id \times \id"'] 
 	\end{tikzcd}.
   \end{center} 
   These two pullbacks are in fact isomorphic as can be confirmed by direct computation.
 \end{remone}

Hence, we managed to translate univalence in terms of a category object internal to $\C$. This is exactly the kind of result we were aiming for. Notice, this result already suggests a different way to characterize univalence via comparison, which requires us to understand the ``completion procedure". 

Let $\DDelta_{\leq 1}$ be the full subcategory of $\DDelta$ with objects $[0]$ and $[1]$, and define the category of graphs as the presheaf category $\Grph=\Fun((\DDelta_{\leq 1})^{op},\set)$. There is a forgetful functor $U:\cat \to \Grph$, that takes a category to the graph consisting of objects, morphisms, source, target and identity map (but forgets the composition). We can now define a functor
\begin{equation} \label{eq:h}
	H: \cat \to \Grph
\end{equation}
that takes a category object $(\cO,\cM,s,t,\comp,\id)$ to the graph $(\cO/\sim,\cM/ \sim,s,t,\id)$, where two objects are equivalent if they are isomorphic and where we have $f \sim f'$ in $\cM$ if there exists a commutative square of the form
\begin{equation} \label{ref:equiv}
	\begin{tikzcd}
		d \arrow[r, "f"] \arrow[d, "\cong"'] & d' \arrow[d, "\cong"] \\
		e \arrow[r, "f'"] & e'
	\end{tikzcd}
\end{equation}
where the vertical morphisms are isomorphisms. We should think of $H(\D)$ as the graph of the category without non-trivial isomorphisms. In fact, for a given category $\D$, there is an evident projection map of graphs $U(\D) \to H(\D)$ that is an isomorphism if and only if the category $\D$ has no non-trivial isomorphisms. 

\begin{remone} \label{rem:hackney example}
	We would hope that the functor $H$ takes categories to categories (instead of just graphs). Unfortunately, this is not the case as composition might not be well-defined on the quotient, as was pointed out to me by Philip Hackney. Let us give an example in the category of finite sets. 
	
	To each morphism of finite sets $f: \{1,..., n\} \to \{1,...,m\}$, we can associate an equation $k_1+...+k_m=n$, where $k_i \geq 0$ is the number of elements in the fiber. Two functions are equivalent if and only if they have the same associated equation and so the equations correspond to our equivalence classes. Composition of two such equation corresponds to adding some of the summands, which is certainly not unique. Let us give one explicit example:
	\begin{itemize}
		\item Let $\alpha: \{1,2,3\} \to \{1,2\}$ be given by $1 \mapsto 1, 2\mapsto 1, 3\mapsto 2$ and $\alpha': \{1,2,3\} \to \{1,2\}$ be given by $1 \mapsto 1, 2\mapsto 2, 3\mapsto 2$ and notice both are in the class $1+2=3$.
		\item Let $\beta: \{1,2,3,4\} \to \{1,2,3\}$ be given by $1 \mapsto 1, 2 \mapsto 2, 3\mapsto 3, 4\mapsto 3$, which is in the class $1+1+2=4$.
		\item The composition $\alpha\beta: \{1,2,3,4\} \to \{1,2\}$ is given by $1 \mapsto 1, 2 \mapsto 1, 3\mapsto 2, 4\mapsto 2$, being in the class $2+2=4$ (which is the equation we get if we add the $1$s in $1+1+2=4$), whereas $\alpha'\beta: \{1,2,3,4\} \to \{1,2\}$ is given by $1 \mapsto 1, 2 \mapsto 2, 3\mapsto 2, 4\mapsto 2$, being in the class $1+3=4$ (which is the equation we get if we add the $1$ and $2$ in $1+1+2=4$), meaning they are not equivalent.
	\end{itemize}
\end{remone}

Coming back to univalence, for a given morphism $p$, we constructed a category object $\C(p) = (B,\cM)$, which gives us a functor 
$$\Hom(-,\C(p)): \C^{op} \to \cat$$
which takes an object $C$ to the category $\Hom(C,\C(p)) = (\Hom(C,B), \Hom(C,\cM))$. Using the construction $H$ and \cref{cor:univalence int cat} we now have following alternative version of univalence. 

 \begin{corone} \label{cor:univalence int cat comption}
	Let $p: E \to B$ be a morphism in a locally Cartesian closed $1$-category $\C$. Then $p$ is univalent if and only if the underlying graph of the category object $\C(p)$ represents the composition functor
	$$H \circ \Hom(-,\C(p)):\C^{op} \to \Grph,$$
	meaning there is a natural isomorphism of graphs
	$$H(\Hom(C,\C(p))) \cong U(\Hom(C,\C(p))).$$
\end{corone}

Before we move on to the $\infty$-categorical case, we will present these results from a different angle. Recall there is a fully faithful functor, the {\it nerve construction}, $N: \cat \to \sset$ \cite{grothendieck1995nerve}, given by 
$$N(\C)_n = \{(f_1:x_0 \to x_1, f_2:x_1 \to x_2, ... , f_n:x_{n-1} \to x_n) \}$$
with boundary maps (face maps) given by composition (inserting identities). Nerves satisfy the {\it Segal condition} (named based on Segal's work on {\it loop spaces} \cite{segal1974loopspaces}), meaning the induced map 
\begin{equation}\label{eq:segalcond}
N(\C)_n \to N(\C)_1 \times_{N(\C)_0} ... \times_{N(\C)_0} N(\C)_1
\end{equation}
is an isomorphism of sets. In fact, this completely characterizes nerves of categories in the sense that a simplicial set is the nerve of a category if and only if it satisfies the Segal condition. Hence, there is an equivalence between {\it Segal sets} and categories. In fact, for a given Segal set $S: \DDelta^{op} \to \set$, we can construct a category 
\begin{equation} \label{eq:cat seg}
	(S_0,S_1,d_1:S_1\to S_0,d_0:S_1 \to S_0,s_0:S_0 \to S_1,d_1:S_2 \to S_1)
\end{equation}
which (in order) gives us objects, morphisms, source, target, identity and composition. The simplicial identities imply that this data in fact gives us a category.

We can generalize this analysis to finitely complete categories. For an internal category object $(B,\cM)$ in a finitely complete category $\C$, we can define a {\it Segal object} $\DDelta^{op} \to \C$, given by $[n] \mapsto \cM \times_B ... \times_B \cM$, the $n$-fold fiber product.

We can use this observation to give an alternative approach to univalence. Starting with a morphism $p:E \to B$, we constructed a category object $\C(p)$ and using the explanation in the previous paragraph, we can now construct a Segal object $\n(p)$. How can we characterize univalence directly in terms of the simplicial object $\n(p)$? By definition we have $\n(p)_3 = \cM \times_B \cM \times_B \cM$ hence we can translate the diagram in \cref{rem:simp diagram} to following span
 \begin{center}
 	\begin{tikzcd}
 	\n(p)_3 \arrow[r, "(d_1d_0\comma d_1d_3)"] &[0.6in] \n(p)_1 \times \n(p)_1 &[0.4in] \n(p)_0 \times \n(p)_0 \arrow[l, "s_0 \times s_0"'] 
 \end{tikzcd}.
 \end{center}
 This gives us another possible characterization of univalence in strict categories.

\begin{corone} \label{cor:univalence segal set}
	Let $\C$ be a locally Cartesian closed $1$-category and $p:E \to B$ be a morphism. Then $p$ is univalent if and only if the Segal object $\n(p)$ has no non-trivial automorphisms, which is equivalent to the following being a pullback square in $\C$
	\begin{center}
		\begin{tikzcd}
			\n(p)_0 \arrow[r] \arrow[d, "\Delta"] & \n(p)_3 \arrow[d, "(d_1d_0\comma d_1d_3)"] \\
			\n(p)_0 \times \n(p)_0 \arrow[r, "s_0 \times s_0"] & \n(p)_1 \times \n(p)_1
		\end{tikzcd}.
	\end{center}
\end{corone}
 A Segal object that satisfies this condition is called {\it complete} \cite[Section 6]{rezk2001css}. Hence we have proven that $p$ is univalent if and only if the Segal object $\n(p)$ is complete. We can think of this result as a translation of univalence from \cref{cor:univalence int cat} from internal categories to Segal objects.
 
  Our last step is to give an analogous result to \cref{cor:univalence int cat comption} for Segal objects, which again requires us to understand the completion operation. Let $S$ be a Segal set, then, by \ref{eq:cat seg}, we know that $S$ is isomorphic to the nerve of a category $S \cong N\C$, so in particular $S_0$ is isomorphic to a set of object and $S_1$ is isomorphic to a set of morphisms. Hence, we can impose an equivalence relation on $S_0$ (where two elements are equivalent if they are isomorphic in $\C$) and we can impose the equivalence relation \ref{ref:equiv} on the set $S_1$. Using the Segal condition (\ref{eq:segalcond}) we can extend this equivalence relation to $S_n$, for $n \geq 1$, hence defining a functor 
  $$H: \Seg\set \to \sset$$
  from the category of Segal sets to simplicial sets, that takes a Segal set $S$ to the simplicial set $H(S)$, with $H(S)_n=S_n/\sim$ and induced boundary and degeneracy maps. Finally, the functor $H$ comes with a projection map $S \to H(S)$, which is an isomorphism of simplicial sets if and only if $S$ has no non-trivial isomorphisms. 
  
  Notice, $H(S)$ is generally not a Segal set. Indeed, we can apply $H$ to the nerve of the category of finite sets and the fact that the quotient graph does not have a well-defined composition (\cref{rem:hackney example}) implies that the map $H(N\FinSet)_2 \to H(N\FinSet)_1 \times_{H(N\FinSet)_0} H(N\FinSet)_1$ is not a bijection of sets, meaning $H(N\FinSet)$ does not satisfy the Segal condition (\ref{eq:segalcond}).

 Finally, for a given Segal object $S_\bullet: \DDelta^{op} \to \C$ in a finitely complete category $\C$, we can define a functor 
 $$\Hom(-,S_\bullet): \C^{op} \to \sset$$
 that takes an object $c$ to the Segal {\it set} $\Hom(c,S_\bullet)$ and using this definition we have following final version of univalence in $1$-categories.
 
 \begin{corone} \label{cor:univalence segal set completion}
  Let $\C$ be locally Cartesian closed $1$-category. Then a morphism $p$ is univalent if and only if the Segal object $\n(p)$ represents the functor 
  $$H \circ \Hom(-,\n(p)): \C^{op} \to \sset,$$
  meaning there is a natural isomorphism of simplicial sets 
  $$H(\Hom(S,\n(p))) \cong \Hom(S,\n(p)).$$ 
 \end{corone}

 Given, the universal property that $H\circ \Hom(-,\n(p))$ satisfies, we can think of $H\circ \Hom(-,\n(p))$ as the {\it completion} of $\Hom(-,\n(p))$. 
 
 \begin{remone} \label{rem:csset}
 	Notice, we would have expected the completion to be a Segal set that is complete (a complete Segal set), however, as we witnessed before this is not always the case. We will see how this will be remedied in the $\infty$-categorical setting.
 \end{remone}
 
 The construction of category objects and Segal objects has helped us internalize the univalence condition in $1$-categories. We would like to generalize this from categories to $\infty$-categories. This is the goal of the coming sections. 
 
 One last note, it is interesting to observe how unnatural the univalence condition is in this setting, as it requires the category in question to have no non-trivial automorphisms (\cref{cor:univalence int cat}) and where the completion does not even give us a category (\cref{cor:univalence int cat comption}). There is an exception where the condition is more natural: $0$-categories!
 
 A category $\C$ is called a {\it $0$-category} if for any two objects there is at most one morphism between them, meaning it is just a set of objects with a relation. Internally, a category object $(\cO,\cM)$ is an {\it internal $0$-category} if the map $\cM\to \cO \times \cO$ is mono in $\C$.  
 Notice that the axioms of a category imply that the relation is {\it reflexive} (the existence of identity maps) and {\it transitive} (the existence of compositions). Meaning that the relation on $\cO$ is always a {\it preorder} \cite{schroder2003preset,schmidt2011relational}. Finally, this preorder has no non-trivial isomorphisms if and only if $x \leq y ,y \leq x$ implies $x =y$, meaning it is {\it anti-symmetric}. A reflexive, anti-symmetric, transitive preorder is simply called a {\it partially ordered set} or {\it poset}. Hence, we have just proven the following result
 
 \begin{corone} \label{cor:complete zero cat}
 	Let $p: E \to B$ be a morphism in $\C$, such that the internal category object $\C(p)$ is an internal $0$-category. Then $p$ is univalent if and only if $\C(p)$ is an internal poset. 
 \end{corone}
 
 So, in the context of $0$-categories, univalence recovers a very common object in mathematics. In fact, we can similarly analyze \cref{cor:univalence int cat comption}. Assume the $1$-category $\C$ is a preorder, which means it is a set $\C$ with a relation $\leq$. Then, the equivalence relation described in \ref{ref:equiv} becomes: For $x,y \in \C$, we have $ x \sim y$ if and only if $x \leq y$ and $y \leq x$. So, the completion $H(\C)$ (\ref{eq:h}) is the quotient poset with respect to this equivalence relation. This is known as the {\it associated poset} of a preorder \cite{facchinifinocchiaro2020poset} and so is in fact a category as well. Hence, we get following improved statement of \cref{cor:univalence int cat comption}.
 
 \begin{corone} \label{cor:complete zero cat completion}
    Let $p: E \to B$ be a morphism in a locally Cartesian closed $1$-category $\C$ such that the category object $\C(p)$ is an internal $0$-category. Then $p$ is univalent if and only if the category object $\C(p)$ represents the composition functor
    $$H \circ \Hom(-,\C(p)):\C^{op} \to \cat,$$
    meaning there is a natural isomorphism of preorders
    $$H(\Hom(C,\C(p))) \cong \Hom(C,\C(p))).$$
 \end{corone}
 
 Hence, at the level of internal $0$-categories univalent morphisms behave very close to our intuition, whereas for internal $1$-categories this intuition breaks down. We will see in the next section that univalence is in fact a much more natural condition in the $\infty$-categorical setting and there is no need to put any restrictions on the dimension of the internal category to get the results we expect. 
 
\section{Univalence via Segal Objects} \label{sec:univ segal obj}
 In the previous section we studied how to relate univalence in a $1$-category to internal conditions: the internal category object we constructed has no non-trivial automorphism (\cref{cor:univalence int cat}) or, equivalently, is isomorphic to its quotient (\cref{cor:univalence int cat comption}). 
 Alternatively, univalence corresponds to the Segal object we constructed being complete (\cref{cor:univalence segal set}) or, equivalently, the Segal object being isomorphic to its completion (\cref{cor:univalence segal set completion}). We would like to generalize this result from categories to $\infty$-categories. However, there are several differences that require our attention.
 
 Up until now the term {\it $\infty$-category} referred to an object in an $\infty$-cosmos equivalent to the $\infty$-cosmos of quasi-categories (\fibref{item:cosmos}). The key insight that allows to use such a general definition is that the category theory of $\infty$-cosmoi has been developed quite extensively by Riehl and Verity \cite{riehlverity2018elements}. Therefore we never needed to restrict our attention to a specific model, such as quasi-categories or complete Segal spaces (both of which are valid examples as explained in \fibref{item:cosmos}). Unfortunately, this observation does not generalize to internal $\infty$-categories. There is (as of yet) no model independent way to define internal $\infty$-categories. Hence, we do have to choose a model.

 The first thing we notice is that we cannot generalize internal category objects to internal $\infty$-categories. Concretely, it would not suffice to specify the data of objects and morphism. Rather in an $\infty$-category (internal or external) composition, and higher compositions, are itself part of the necessary data. Hence, we cannot define internal $\infty$-categories just via two objects and an appropriate choice morphisms, which means there is no hope that we can state generalizations of \cref{cor:univalence int cat}/\cref{cor:univalence int cat comption}.
 Fortunately, Segal objects do in fact generalize to finitely complete $\infty$-categories. Thus, we will henceforth focus on generalizing the results in \cref{cor:univalence segal set}/\cref{cor:univalence segal set completion} to $\infty$-categories.

 Hence, our goal is to show that given a morphism $p: E \to B$ in a nice enough $\infty$-category $\C$, we can construct a Segal object with $0$-level $B$ and $1$-level $\morphism{E \times B}{B \times E}_{B \times B}$, which leads to our main technical challenge! The non-strictness of $\infty$-categories poses a serious challenge when trying to construct simplicial objects, as it involves infinitely many coherence conditions. 
 One way to avoid this issue is to try to strictify our constructions and use strict categories with homotopical data. This has been done very effectively in \cite{stenzel2019univalence}, where Stenzel uses model categories and relative categories \cite{barwickkan2012relativecategory} to study univalence. However, we will avoid this approach for several reasons. First  of all, although relative categories are a model of $(\infty,1)$-categories, they are not an $\infty$-cosmos. In particular, the category theory of relative categories has received far less attention that quasi-categories or even complete Segal spaces. However, one important goal of this work is to illustrate how univalence can relate to topics of interest among higher category theorists, as we will illustrate in \cref{sec:univ and topos}/\cref{sec:examples}.

 Hence, we will stay inside an $\infty$-categorical setting that avoids strictifying the $\infty$-category. Rather we use the fibrations of $\infty$-categories to construct the desired Segal object. 
 Unfortunately, as one might expect, making this idea precise is quite technical and relies on the theory of {\it representable Cartesian fibrations} \cite{rasekh2017cartesian}. Hence, the proof has its own section (\cref{sec:proof}). In this section we introduce the definitions that are needed to state the main theorem and some important implications, relegating the technical definitions and proofs to \cref{sec:proof}.
 
 \begin{defone} \label{def:seg obj}
 	Let $\C$ be a finitely complete $\infty$-category.
   A simplicial object $W: \DDelta^{op} \to \C$ is a {\it Segal object} if for all $n \geq 2$ the induced map 
  $$W_n \to W_1 \times_{W_0} ... \times_{W_0} W_1$$
  is an equivalence in $\C$.	
 \end{defone}
 A Segal object satisfies many desirable properties of an internal $\infty$-category: It has objects, morphisms, equivalences, mapping spaces composition ... \cite[Section 3]{rasekh2017cartesian}. However, unlike in the context of $1$-categories, it is in fact not an internal $\infty$-category. 
 For example we would like a fully faithful and essentially surjective functor to be an equivalence, however, this fails for Segal objects, as can be seen in the case Segal objects in the $\infty$-category of spaces, so-called Segal spaces \cite[Theorem 7.7]{rezk2001css}. The fix is to add the {\it completeness condition}.
 \begin{defone}\label{def:cso}
 	Let $\C$ be a finitely complete $\infty$-category.
 	A Segal object $W: \DDelta^{op} \to \C$ is a {\it complete Segal object} if the square 
 	\begin{center}
 		\begin{tikzcd}
 			W_0\arrow[r] \arrow[d, "\Delta"] & W_3 \arrow[d, "(d_1d_0\comma d_1d_3)"] \\
 			W_0 \times W_0 \arrow[r, "s_0 \times s_0"] & W_1 \times W_1
 		\end{tikzcd}
 	\end{center} 
 	is a pullback square in $\C$.
 \end{defone} 
 Continuing our discussion above, fully faithful and essentially surjective functors of complete Segal objects are in fact equivalences \cite[Theorem 3.13]{rasekh2017cartesian} and complete Segal objects in spaces, {\it complete Segal spaces}, do in fact give us $\infty$-categories  \cite{toen2005unicity,rezk2001css,joyaltierney2007qcatvssegal,bergner2007threemodels,riehlverity2018elements} and so we consider complete Segal objects as our model of {\it internal $\infty$-categories}. 
 
  In \cref{cor:univalence segal set completion} we used the fact that every Segal object in a strict $1$-category gives us a functor valued in simplicial sets. We want to generalize this construction to Segal objects in $\infty$-categories, this requires us dealing with the technicalities of {\it representable Cartesian fibrations}, which generalize representable right fibrations (\fibref{item:rep right fib}). Using representable Cartesian fibrations, we can prove that for every Segal object $W: \DDelta^{op} \to \C$, there exists a functor that takes an object $X$ to the Segal space
  \begin{center}
  	\simpset{\Map_\C(X,W_0)}{\Map_\C(X,W_1)}{\Map_\C(X,W_2)}{}{}{}{}	
  \end{center}
  which we simply denote by $\Map(X,W)$. In particular, we can use a complete Segal space to construct a Cartesian fibration $\C_{/W} \to \C$. Every Cartesian fibration equivalent to one of this form is called {\it representable}. For more details on this construction see \cref{sec:proof}.
  
  \begin{remone} \label{rem:css v inf}
  	Given that until now our theory of $\infty$-categories relied on an $\infty$-cosmos biequivalent to quasi-categories (\fibref{item:cosmos}), the fact that we are using complete Segal spaces requires some explanation. There is a biequivalence between the $\infty$-cosmos of complete Segal spaces and our theory of $\infty$-categories given via the nerve and underlying construction, which is explained in more detail in \fibref{item:css}. Given this equivalence, in order to avoid adding additional notation we will keep denoting the underlying complete Segal space corresponding to an $\infty$-category by $\C$ and will simply refer to it by the complete Segal space.
  \end{remone}

  Finally, fix a morphism $p: E \to B$. We define the Cartesian fibration $\sO_\C^p \to \C$ as the full sub-Cartesian fibration of the target fibration $\sO_\C$ consisting of morphisms that can be obtained as a pullback of $p$. We denote the fiber of $\sO_\C^p \to \C$ over an object $X$ by $(\C_{/X})^{(p)}$ and note it is a full subcategory of the over-category $\C_{/X}$.
  
 We can now state our main theorem relating univalence to Segal objects. 
 
 \begin{theone} \label{the:main theorem}
 	Let $\C$ be a finitely complete $\infty$-category and $p: E \to B$ be a morphism such that pulling back along $E \times B \to B \times B$ has a right adjoint.
 	\begin{enumerate}
 		\item There exists a (homotopically) unique Segal object $\n(p)$ such that $\n(p)_0 \simeq B$ and $\n(p)_1 \simeq \morphism{(\pi_1)^*E}{(\pi_2)^*E}_{B \times B}$. 
 		\item For an object $X$, there exists a natural map of Segal spaces 
 		$$\Map_\C(X,\n(p)) \to \C_{/X},$$
 		where $\C_{/X}$ denotes the underlying complete Segal space (\cref{rem:css v inf}). The essential image is $(\C_{/X})^{(p)}$, the full sub-$\infty$-category consisting of morphisms that can be obtained as a pullback of $p$.
 		\item The following are equivalent.
 		\begin{enumerate}
 			\item $p$ is univalent.
 			\item {\bf Representability:} The Cartesian fibration $\sO^p_\C \to \C$ is represented by $\n(p)$.
 			\item {\bf Univalent Completion:} The map $\Map_\C(X,\n(p)) \to (\C_{/X})^{(p)}$ is a natural equivalence of Segal spaces for every object $X$.
 			\item {\bf Completeness:} The Segal object $\n(p)$ is complete.
  		\end{enumerate} 
 	\end{enumerate}
 \end{theone}
 
 The proof of this statement is the goal of \cref{sec:proof}. For the remainder of this section we want to look at some interesting implications of this result and in particular how the various equivalent conditions give us different perspectives on univalence. 
 
 {\bf Representability:}
 In \cref{prop:univ final} we proved that if $p$ is univalent then there is an inclusion of right fibrations $\C_{/B} \hookrightarrow \Oall_\C$. The essential image of this inclusion is precisely the morphisms that can be obtained as a pullback of $p$. Hence, according to \cref{prop:univ final}, univalence implies that $B$ represents the full sub-right fibration of $\Oall_\C$ consisting of morphisms that are pullbacks of $p$.
 
 From that perspective the representability condition for the Segal object $\n(p)$ can be seen as a direct lift of the statement in \cref{prop:univ final} from right fibrations to Cartesian fibrations. We can equivalently say that the representability of the Cartesian fibration is completely determined by the representability of its underlying right fibration (\fibref{item:underlyng right fib}).
 This result might seem surprising as usually $\infty$-categorical properties are not just determined by their objects. However, the data in question is that of a full subcategory, which in fact is completely determined by a choice of objects.
  
 We can use this observation to generalize another part of \cref{prop:univ final}.
 
  \begin{corone}
  	Let $\C$ be finitely complete $\infty$-category and $p: E \to B$ be a morphism such that pulling back along $E \times B \to B \times B$ has a right adjoint. Then $p$ is univalent if and only if the map $\C_{/\n(p)} \to \sO_\C$ is an inclusion of Cartesian fibrations.
  \end{corone}

 Finally, note this means we can use \cref{the:main theorem} to give an alternative (far more complicated) proof of \cref{prop:univ final}. 
 
 {\bf Univalent Completion:}
 The fibration $\sO_\C^{(p)} \to \C$ is complete by its very construction and so we should think of the inclusion  $\Map(-,\n(p)) \to (\C_{/-})^{(p)}$
 as a ``completion functor" that takes a Segal space to its universal complete Segal space, where universality follows from the fact that $\n(p)$ is complete if and only it the map is an equivalence. This relation between Segal objects and their univalent completions has been studied in the context of relative categories (and in particular model categories) by Stenzel \cite[Section 5]{stenzel2019univalence} and in particular \cite[Proposition 5.16, Theorem 5.18]{stenzel2019univalence}.
 
 It is interesting to observe the contrast between the $1$-categorical completions (\cref{rem:csset}) and the $\infty$-categorical completion. In the $1$-categorical setting, the completion was not guaranteed to actually be a category (or Segal set) (\cref{rem:hackney example}). We had to restrict to $0$-categories to recognize the univalent completion as the associated poset (\cref{cor:complete zero cat completion}). This difference is strongly related to the fact that in the $\infty$-categorical setting univalence corresponds to the existence of certain internal $\infty$-categories (as explained below), whereas in the $1$-categorical setting univalence corresponds to the existence of internal $1$-categories without non-trivial automorphisms (\cref{cor:univalence int cat}), which is quite restrictive, and only in the $0$-categorical case to the existence of posets (\cref{cor:complete zero cat}).
 
 We do not expect univalent completions to exist in every $\infty$-category, however, do they exist in many $\infty$-categories we care about. 
 In fact, with the benefit of hindsight, we can think of the {\it completion functor} for Segal spaces due to Rezk \cite[Section 14]{rezk2001css} as a proof that completions always exist in the $\infty$-category of spaces. This was in fact also proven by van den Berg and Moerdijk \cite[Theorem 5.1]{bergmoerdijk2018univalentcompletion} using the definition of univalence due to Kapulkin and Lumsdaine (motivated by work of Voevodsky) \cite{kapulkinlumsdaine2012kanunivalent}.
 
 This observation can be vastly generalized. Concretely, let us denote $\Seg(\C)$ be the $\infty$-category of Segal objects (taken as a full sub-category of all simplicial objects in $\C$) and let $\CSS(\C)$ be the full sub-$\infty$-category of complete Segal objects. Then we have following result.
 
  \begin{corone}
 	Let $\C$ be an $\infty$-category such that the inclusion functor $\CSS(\C) \hookrightarrow \Seg(\C)$ has a left adjoint $\Comp: \Seg(\C) \to \CSS(\C)$ (for example if $\C$ is a presentable quasi-category \cite[Corollary 5.5.2.9]{lurie2009htt}). Then, for every given morphism $p$, $\Comp(\n(p))$ is the completion of $\n(p)$, meaning it represents $\sO_\C^{(p)} \to \C$.
 \end{corone}
 
 Having an internal completion allows us to reformulate univalence as follows.
 
 \begin{corone}
 	Let $\C$ be finitely complete $\infty$-category and $p: E \to B$ be a morphism such that pulling back along $E \times B \to B \times B$ has a right adjoint. Assume the Cartesian fibration $\sO_\C^{p} \to \C$ is representable by a complete Segal object $\overline{\n(p)}$, which comes with a map 
 	$\n(p) \to \overline{\n(p)}$. Then $p$ is univalent if and only if this map is an equivalence. 
\end{corone}
 
  There are other environments where univalent completions have been studied as well.
  For example, the construction has been adapted to the setting of homotopy type theory under the name  {\it univalent completions}, {\it Rezk completions} or {\it stack completion}, where it is used to construct univalent categories out of precategories \cite{ahrenskapulkinshulman2015rezkcompletion}, \cite[Section 9.9]{hottbook2013}. 
  
  Another interesting and maybe more unexpected example is the setting $(\infty,n)$-categories. In \cite{ayalafrancis2018flagged} Ayala and Francis define {\it flagged $\infty$-categories} as a model independent approach to Segal spaces and give a construction of univalent completion in their setting, which results in an $(\infty,n)$-category. 
 
 {\bf Internal $\infty$-Categories:}
 At the end of \cref{sec:strict cat} we observed that the univalence condition seems to be oddly out of place; it imposes an additional condition on an already well-defined category object. It was only in the $0$-categorical case where univalence seemed to have a meaning, as in that case univalence corresponded to the internal preorder being an internal poset (\cref{cor:complete zero cat}) or a preorder being isomorphic to its associated poset (\cref{cor:complete zero cat completion}). 
 
 In the $\infty$-categorical setting, on the other side, the univalence condition is of utmost importance, as completeness is a necessary condition to define internal $\infty$-categories, without any restriction on the internal $\infty$-category. As explained before, Segal objects are in fact not enough to get internal $\infty$-categories. This observation, along with \cref{the:main theorem}, gives us following result.
 
 \begin{corone} \label{cor:univ cso}
 	 Let $\C$ be a finitely complete $\infty$-category and let $p:E \to B$ be a morphism such that pulling back along $E \times B \to B \times B$ has a right adjoint. Then the resulting simplicial object $\n(p)$ is an internal $\infty$-category if and only if $p$ is univalent.
 \end{corone}
 
 In addition to that, this result also addresses the concern that arose at the end of \cref{sec:univ target} that tried to give internal meaning to \cref{cor:q} by describing univalence of a morphism as completeness of the associated Segal object. 
 From this perspective, \cref{lemma:q} can be understood as giving an explicit description of $\n(p)_{hoequiv}$ without having constructed (or understood) the Segal object. Note, this perspective on univalence via completeness has already been studied in the context of logical model categories \cite{arndtkapulkin2011modelstypetheory} by Stenzel \cite[Theorem 4.5, Corollary 4.6]{stenzel2019univalence}.
 
 Notice the opposite does not hold, meaning we cannot characterize completeness of a Segal object via univalence of a certain morphism. Indeed, not every Segal object can be obtained as the nerve of a morphism. For example, in the category of sets, the complete Segal objects are precisely the small categories without automorphisms (as discussed in \cref{sec:strict cat}). On the other hand, as we shall see in \cref{lemma:univ set}, there are only four univalent morphisms in the category of sets whose nerves give us the poset $\{ 0 \leq 1 \}$ and its three subcategories.
 
 We end this section with a noteworthy observation. 
 Given that the completeness condition was motivated by the desire to guarantee that fully faithful and essentially surjective functors are equivalences and the connection between completeness and univalence, we could ask ourselves whether there is an interesting way to understand univalence from this perspective. We will not further pursue this question in the context of this work.

 \section{Proof of the Main Theorem} \label{sec:proof}
  This section is dedicated to proving \cref{the:main theorem}. The reader who wants to avoid the technical details is encouraged to move on to \cref{sec:univ and topos}, where we look at connections to topos theory, or \cref{sec:examples}, where we look at further examples.
  
 Up until this section we used Cartesian fibrations and right fibrations to study univalence (\fibref{item:cart rfib}). Unfortunately, as we shall explain, the theory of Cartesian fibrations does not suffice anymore. The main step in the proof of \cref{the:main theorem} is to construct a Segal object $\n(p)$ out of a morphism $p: E\to B$. We want to construct this Segal object using representability, meaning we want to construct a fibration over $\C$ such that the fiber over an object $c$ is the Segal space 
 \begin{center}
 	\simpset{\Map_\C(\n(p)_0,c)}{\Map_\C(\n(p)_1,c)}{\Map_\C(\n(p)_2,c)}{}{}{}{}.
 \end{center}
 A fibration with such fiber is usually not a Cartesian fibration and in fact we want to conclude, as stated in \cref{the:main theorem}, that it is only Cartesian if $p$ is univalent. We need to construct fibrations fibered in Segal spaces as a direct generalization of Cartesian fibrations, which are fibered in $\infty$-categories, and that would allow us to study presheaves valued in Segal spaces. In fact, the whole proof in this section could be summarized as using various fibration and strictification arguments that reduce the proof to \cref{lemma:main theorem one rep}, which focuses on the object of morphisms, which also played the key role in the strict case (\ref{eq:m}). The relevant theory of fibrations has been introduced in \cite{rasekh2021cartfibcss,rasekh2021cartfibmarkedvscso,rasekh2017cartesian} and we will review the relevant results here.

A {\it Reedy right fibration over $\C$} \cite[Section 2]{rasekh2021cartfibcss} is a simplicial right fibration over $\C$, $\R_\bullet: \DDelta^{op} \to \RFib_{/\C}$ (where $\RFib_{/\C}$ was defined in \fibref{item:cart rfib}). Restricting to the object $[n]$ in $\DDelta$ gives us functors
$$\RFib_n: \Fun(\DDelta^{op},\RFib_{/\C}) \to \RFib_{/\C}$$
that takes a Reedy right fibration to the $n$-th level of our simplicial right fibration.

We say a Reedy right fibration $\R_\bullet$ is a {\it Segal Cartesian fibration} \cite[Section 4]{rasekh2021cartfibcss} if it satisfies the Segal conditions, meaning for $n \geq 2$ we have an equivalence of right fibrations 
$$\RFib_n(\R) \xrightarrow{ \ \simeq \ } \RFib_1(\R) \times_{\RFib_0(\R)} ... \times_{\RFib_0(\R)} \RFib_1(\R).$$
Moreover, for a Segal Cartesian fibration $\R$ define the {\it right fibration of equivalences}, $\R_{hoequiv}$, via following pullback diagram 
\begin{equation} \label{eq:r hoequiv}
 \begin{tikzcd}
 	\R_{hoequiv} \arrow[r] \arrow[d] & \RFib_3(\R) \arrow[d, "(d_1d_0\comma d_1d_3)"] \\
 	\RFib_0(\R) \times_\C \RFib_0(\R) \arrow[r, "s_0 \times s_0"] & \RFib_1(\R) \times_\C \RFib_1(\R) 
 \end{tikzcd}.
\end{equation}
We say a Segal Cartesian fibration $\R$ is a {\it Cartesian fibration} if it is {\it complete}, meaning the map $\RFib_0(\R) \to \R_{hoequiv}$ over $\RFib_0(\R) \times_\C \RFib_0(\R)$ induced by the map $\Delta: \RFib_0(\R) \to \RFib_0(\R) \times_\C \RFib_0(\R)$ is an equivalence of right fibrations. 
Initially this terminology might seem confusing given that we already have a notion of Cartesian fibration. However, we will show in \ref{eq:rfib} that these notions coincide appropriately.

Our key interest in simplicial right fibrations is because of {\it representability}. Let $W: \DDelta^{op} \to \C$ be a simplicial object in $\C$. Then there exists a Reedy right fibration $\C_{/W_\bullet} \to \C$, which at level $n$ is equivalent to the representable right fibration $\C_{/W_n} \to \C$ (\fibref{item:rep right fib}) \cite[Proposition 2.4]{rasekh2017cartesian}. Moreover, there is a {\it Yoneda lemma for representable Reedy right fibrations}.

\begin{propone}\label{prop:yoneda reedy right}
	\cite[Theorem 2.7]{rasekh2017cartesian}
	Let $W,V: \DDelta^{op} \to \C$ be two simplicial objects. Then there is an equivalence of mapping spaces 
	$$\Map_{\Fun(\DDelta^{op},\RFib_{/\C})}(\C_{/W},\C_{/V}) \simeq \Map_{\Fun(\DDelta^{op},\C)}(W,V).$$
	So, in particular, a morphism of Reedy right fibrations is an equivalence if and only if the corresponding map of fibrations is an equivalence. 
\end{propone}
$W$ is a Segal object (\cref{def:seg obj}) if and only if $\C_{/W_\bullet} \to \C$ is a Segal Cartesian fibration \cite[Proposition 3.4]{rasekh2017cartesian}. By the same result, $W$ is a complete Segal object (\cref{def:cso}) if and only if $\C_{/W} \to \C$ is a Cartesian fibration. Hence, whenever we have a Segal object we can construct a Segal Cartesian fibration out of it. We want a result that helps us construct a Segal object out of a Segal Cartesian fibration. For that we have following result.

\begin{propone} \label{prop:rep seg fib}
	\cite[Theorem 4.3]{rasekh2017cartesian}
	Let $\C$ be an $\infty$-category with finite limits. 
	Let $\R \to \C$ be a Segal Cartesian fibration such that $\RFib_0(\R) \to \C$ and $\RFib_1(\R) \to C$ are both representable right fibration. Then there exists a Segal object $W: \DDelta^{op} \to \C$ such that the right fibration $\RFib_n(\R) \to \C$ is represented by $W_n$.
	Moreover, $W$ is a complete Segal object if and only if $\R\to \C$ satisfies the completeness condition.
\end{propone}

We now have reviewed all the relevant concepts regarding representable Segal Cartesian fibrations, which we need to prove \cref{the:main theorem}.
Recall that $\sO_\C^p \to \C$ denotes the full sub-Cartesian fibration of $\sO_\C \to \C$ consisting of morphisms that can be obtained as a pullback of $p$.
\begin{proof}[Proof of \cref{the:main theorem}] 
	{\it (1)}
	First we construct a new Reedy right fibration out of the Cartesian fibration $\sO_\C^p \to \C$, that we denote $\sO_\C^{p/B} \to \C$	 (\cref{lemma:main theorem fib exists}). We then show that $\sO_\C^{p/B} \to \C$ is a Segal Cartesian fibration  \cref{lemma:main theorem fib Segal}), such that $\RFib_0(\sO_\C^{p/B}) \to \C$ is represented by $B$ (\cref{lemma:main theorem fib exists}) and that $\RFib_1(\sO_\C^{p/B}) \to \C$ is represented by $\morphism{(\pi_1)^*E}{(\pi_2)^*E}_{B \times B}$ (\cref{lemma:main theorem one rep}). By \cref{prop:rep seg fib}, this means that  $\sO_\C^{p/B} \to \C$ is represented by a Segal object we denote by $\n(p)$. 
	
	{\it (2)} 
	By construction of $\sO_\C^{p/B}$ we have a map of Segal Cartesian fibrations $\sO_\C^{p/B} \to \sO_\C$ that we observe has essential image $\sO_\C^p$ (\cref{lemma:main theorem fib exists}). 

	{\it (3)} We will cycle through the conditions:
	
	{\it (a) $\Rightarrow$ (b)} If $p$ is univalent, then the map $\C_{/B} \to \RFib_0(\sO^p_\C)$ is an equivalence of right fibrations (\cref{prop:univ final}) and so we have an equivalence of Cartesian fibrations $\sO^{p/B}_\C \to \sO^p_\C$. By definition, the first Cartesian fibration is represented by $\n(p)$ and so by the equivalence the second one must be as well.
	
	{\it (b) $\Rightarrow$ (c)} By assumption, $\n(p)$ represents $\sO^p_\C$, which gives us an equivalence $\C_{/\n(p)} \to \sO^p_\C$. The result now follows from comparing fibers.
	
	{\it (c) $\Rightarrow$ (d)} By assumption, we have an equivalence $\C_{/\n(p)} \to \sO^p_\C$ and $\sO^p_\C$ is a Cartesian fibration, which implies that $\C_{/\n(p)}$ is a Cartesian fibration, which, by \cite[Proposition 3.4]{rasekh2017cartesian}, implies that $\n(p)$ is a complete Segal object.
	
	{\it (d) $\Rightarrow$ (a)} If $\n(p)$ is complete, then $\sO^{p/B}_\C\to \C$ is a Cartesian fibration, which, by \cref{lemma:main theorem completeness}, implies that $\C_{/B} \to \Oall_\C$ is an inclusion, which, by \cref{prop:univ final}, implies that $p$ is univalent.
\end{proof}
  Before we move on to prove all the necessary lemmas, we want to use \cref{the:main theorem} and the steps in its proof to give an alternative proof to \cref{prop:univ if mono}, giving an example how Segal objects can be used to give internal proofs and also give us a better understanding of univalence. 
  
  \begin{propone} \label{prop:univ if mono alt}
  	Let $\C$ be locally Cartesian closed.
  	Let $p: E \to B$ be univalent and assume we have following pullback square 
  	\begin{center}
  		\begin{tikzcd}
  			C \arrow[r] \arrow[d, "q"]  & E \arrow[d, "p"] \\
  			D \arrow[r] & B
  		\end{tikzcd}.
  	\end{center}
  	Then $q$ is univalent if and only if $D \to B$ is mono
  \end{propone}
  
  \begin{proof}
  	The pullback diagram
  	gives us a map of right fibrations $\C_{/D} \to \C_{/B}$ over $\C$ (via the Yoneda lemma \fibref{item:rep right fib}). Using the construction of the Segal Cartesian fibration $\sO^{p/B}_\C$ (\cref{lemma:main theorem fib exists}), we get a map of Segal Cartesian fibration fibrations $\sO^{q/D}_\C \to \sO^{p/B}_\C$ over $\C$.
  	According to \cref{the:main theorem}, $\sO^{p/B}_\C \to \C$ is a Cartesian fibration, meaning $(\sO^{p/B}_\C)_{hoequiv}\simeq \C_{/B}$. 
  	Following the argument in the proof of \cref{lemma:main theorem completeness},  we can evaluate $(\sO^{q/D}_\C)_{hoequiv} $ as the following pullback square 
  	\begin{center}
  		\begin{tikzcd}
  			(\sO^{q/D}_\C)_{hoequiv} \arrow[r] \arrow[d] & \C_{/B}\arrow[d] \\
  			\C_{/D} \times_\C \C_{/D} \arrow[r] & \C_{/B} \times_\C \C_{/B} 
  		\end{tikzcd}.
  	\end{center}
    However, by definition given in \ref{eq:r hoequiv}, $\sO^{q/D}_\C \to \C$ is Cartesian if and only if $\C_{/D} \to (\sO^{q/D}_\C)_{hoequiv}$ is an equivalence. 
    This, by \fibref{item:over cat trunc}, is equivalent to $\C_{/D} \to \C_{/B}$ being fully faithful, which, again by \fibref{item:over cat trunc}, is equivalent to $D \to B$ being mono.
  \end{proof}
  
   We will now embark on a long journey and complete all the steps and lemmas used in the proof of \cref{the:main theorem}.
  The first steps is to understand the behavior of Cartesian fibrations as we change models of $\infty$-categories, which is the goal of the following remark.
  
  \begin{remone} \label{rem:model cart}
   Up until this point out our theory of $\infty$-categories has been an $\infty$-cosmos $\K$ biequivalent to quasi-categories (\fibref{item:cosmos}), for which we denote the very large (\fibref{item:universes}) quasi-category of Cartesian fibrations over $\C$ by $\Cart_{/\C}$ (\fibref{item:cart rfib}) . Our next construction involves a change of models and so in order to avoid any confusion, we will denote the quasi-category of Cartesian fibrations over $\C$ for a theory of $\infty$-categories $\K$ by $\Cart^\K_{/\C}$. Notice the biequivalence of $\infty$-cosmoi $((-)_0, 1 \otimes -)$ gives us an equivalence of quasi-categories $\Cart^\K_{/\C} \simeq \Cart^{\Q\cat}_{/\C_0}$.
  
  Now, restricting our attention to  quasi-categories, \cite[Subsection 3.2]{rasekh2021cartfibmarkedvscso} states that there is a fully faithful functor of quasi-categories 
  \begin{equation} \label{eq:rfib qcat}
  	\Cart^{\Q\cat}_{/\C} \to \Fun(\DDelta^{op},\RFib^{\Q\cat}_{/\C})
  \end{equation}
  induced by a zig-zag of Quillen equivalences. Moreover, a simplicial right fibration is in the essential image if and only if it satisfies the Segal and completeness condition, which justifies the terminology Cartesian fibration for the essential image (\ref{eq:r hoequiv}).
  If $\C=\Delta^0$, the final quasi-category, then we have equivalences $\RFib^{\Q\cat}_{/\Delta^0} \simeq \widehat{\s}$, $\Cart^{\Q\cat}_{/\Delta^0} \simeq \widehat{\cat}_\infty$ (\fibref{item:rfib over space}) and the functor  
  $$\RFib^{\Q\cat}: \cat_\infty \to \Fun(\DDelta^{op},\widehat{\s})$$
  takes a quasi-category $\C$ to the complete Segal space $\RFib(\C)$, given at level $n$ by $(\C^{\Delta^n})^\simeq$, the maximal sub-$\infty$-groupoid of $n$-simplices (\fibref{item:underlyinggroupoid}). This follows from the fact that the construction in \cite[Subsection 3.2]{rasekh2021cartfibmarkedvscso} fiber-wise reduces to the functor $t^!$ in \cite{joyaltierney2007qcatvssegal} (also denoted nerve in \cite[Example 1.3.9]{riehlverity2018elements}), which has precisely that description. Hence, we can think of $\RFib$ in \ref{eq:rfib qcat} as a fibrational generalization of the construction in \cite{joyaltierney2007qcatvssegal}.
  
  We can now compose this functor in \ref{eq:rfib qcat} in an  appropriate way, using the equivalences given in \fibref{item:cosmos}, to generalize it to $\K$ 
  \begin{equation} \label{eq:rfib}
  	\begin{tikzcd}
  		\RFib^\K: \Cart^{\K}_{/\C} \arrow[r, "\simeq"] & \Cart^{\Q\cat}_{/\C_0} \arrow[r, hookrightarrow] & \Fun(\DDelta^{op},\RFib^{\Q\cat}_{/\C_0}) \arrow[r, "\simeq"] & \Fun(\DDelta^{op},\RFib^{\K}_{/\C})
  	\end{tikzcd}.
  \end{equation}
  and notice that $\RFib$ satisfies analogous results, in particular it is fiber-wise characterized as taking $\infty$-categories to a simplicial object valued in discrete $\infty$-categories, which is the underlying complete Segal space (\fibref{item:css}). 
  We will henceforth use the resulting functor $\RFib^\K:\Cart^{\K}_{/\C} \to \Fun(\DDelta^{op},\RFib^{\K}_{/\C})$ only in the context of our theory of $\infty$-categories $\K$ again and hence drop the $\K$ from the notation.
 \end{remone}

  Given the explanation above, we realize that $\RFib_0$ takes a Cartesian fibration to the underlying right fibration (\fibref{item:underlyng right fib}). Given this intuition we call $\RFib_0: \Fun(\DDelta^{op},\RFib_{/\C}) \to \RFib_{/\C}$ the {\it underlying right fibration functor}, even if the fibrations is not Cartesian.
  Notice, the functor $\RFib_0$ has a left and right adjoint, which we can describe explicitly. 
  The left adjoint of $\RFib_0$ is the functor
  $$\Disc:\RFib_{/\C} \to \Fun(\DDelta^{op},\RFib_{/\C})$$
  that takes a right fibration to the constant simplicial object. The naming should remind us that it is a {\it discrete simplicial object}.
  On the other hand, the right adjoint 
  \begin{equation} \label{eq:cech functor}
  	\Cech:\RFib_{/\C} \to \Fun(\DDelta^{op},\RFib_{/\C})	
  \end{equation}
  takes a right fibration $\R\to \C$ to the Reedy right fibration $\Cech(\R) \to \C$ defined as $\RFib_n(\Cech(\R)) = \R^n$, where we are taking the product in $\RFib_{/\C}$, and where the simplicial operators are given by projection and diagonal. The construction resembles the construction of the {\it {\v C}ech nerve} \cite{alexandroff1928cechnerve} and should make us think of a category in which there is a unique morphism between any two objects. We will need a more detailed understanding of the $\Cech$ construction.
  
   \begin{lemone} \label{lemma:main theorem constant fib}
  	Let $\R \to \C$ be a right fibration. Then $\Cech(\R)$ is a Segal Cartesian fibration with $\Cech(\R)_{hoequiv} \simeq \RFib_1(\Cech(\R)) = \R \times_\C\R$.
  	Moreover, if $\R$ is represented by $B$, then $\Cech(\R)$ is represented by a Segal object $\Cech(B)$ given level-wise by $\Cech(B)_n = B^{n+1}$.
  \end{lemone} 
  
  \begin{proof}
  	The first part is a direct computation. Using an induction argument it suffices to confirm the Segal condition for $\RFib_2(\Cech(\R))$. By definition $d_2:\RFib_2(\Cech(\R)) = \R \times_\C\R\times_\C \R \to \R \times_\C \R=\RFib_1(\Cech(\R))$ is given by projecting onto the first two coordinates. Similarly, $d_2:\RFib_2(\Cech(\R)) = \R \times_\C\R\times_\C \R \to \R \times_\C \R=\RFib_1(\Cech(\R))$ projects unto the second two coordinates. It now follows directly that the following square is a pullback
  	\begin{center}
  		\begin{tikzcd}
  			\R \times_\C\R\times_\C \R  \arrow[r, "d_0"] \arrow[d, "d_0"] & \R \times_\C \R \arrow[d, "\pi_1=d_1"] \\
  			\R \times_\C \R \arrow[r, "\pi_2=d_0"] & \R 
  		\end{tikzcd}
  	\end{center}
  	confirming the Segal condition. 
  	
  	We move on to compute $\Cech(\R)_{hoequiv}$. By direct computation the map $d_1d_3: \RFib_3(\Cech(\R)) \to \RFib_1(\Cech(\R))$ is given by projecting onto the first and third coordinate. Similarly, the map $d_1d_0:\RFib_3(\Cech(\R)) \to \RFib_1(\Cech(\R))$ is given by projecting on the second and fourth coordinate. Hence, the map $(d_1d_0,d_1d_3): \R \times_\C \R \times_\C \R \times_\C \R \to  \R \times_\C \R \times_\C \R \times_\C \R$ is just a reshuffling of the coordinates and thus an equivalence. Pulling back this equivalence along the map 
  	$\Delta \times \Delta: \R \times_\C\R \to  \R \times_\C \R \times_\C \R \times_\C \R $ gives us the desired result.
  	
  	Finally, if $\R$ is represented by $B$ then we have an equivalence $\R \simeq \C_{/B}$ (\fibref{item:rep right fib final}) and so we get 
  	$$\RFib_n(\Cech(\R)) \simeq \RFib_n(\Cech(\C_{/B})) \simeq \C_{/B} \times_\C ... \times_\C \C_{/B} \simeq \C_{/B \times ... \times B},$$
  	where the last step follows from \fibref{item:rep right fib limit}. 
  \end{proof}
 
    We now have developed all the necessary background to prove the desired lemmas. Recall $\sO_\C\to \C$ is the target Cartesian fibration and $\sO^p_\C \to \C$ is the sub-Cartesian fibration consisting of morphisms that can be obtained as the pullback of $p$.
 
 \begin{lemone} \label{lemma:main theorem fib exists}
 	There exists a Reedy right fibration $\sO_\C^{p/B}$ with a map of Reedy right fibrations $\sO_\C^{p/B} \to \sO_\C$ which has essential image $\sO_\C^p$. Moreover, the underlying right fibration of $\sO^{p/B}_\C$ is $\C_{/B}$.
 \end{lemone}
 
\begin{proof}
	Recall the map $p$ gives us a map of right fibrations $\C_{/B} \to \Oall_\C$ (\fibref{item:rep right fib}), which gives us a map on codiscrete objects $\Cech(\C_{/B}) \to \Cech\RFib_0(\sO_\C)$. Now, define the Reedy right fibration $\sO^{p/B}_\C$ via following pullback square of right fibrations over $\C$
	\begin{center}
		\begin{tikzcd}
			\sO^{p/B}_\C \arrow[r] \arrow[d] & \sO_\C \arrow[d, "c"] \\ 
			\Cech(\C_{/B}) \arrow[r] & \Cech\RFib_0(\sO_\C)
		\end{tikzcd}.
	\end{center}
    Here $c$ is the counit of the adjunction (described in \ref{eq:cech functor}).
    An object is in the image if it can be obtained as a pullback of $p$, which is precisely how we define $\sO_\C^p$, meaning $\sO_\C^p$ is precisely the image of the map $\sO^{p/B}_\C \to \sO_\C$. 
    Finally, notice $\RFib_0(c): \RFib_0(\sO^p_\C) \to \RFib_0(\Cech\RFib_0(\sO^p_\C))$ is simply the identity (by the triangle identity) and so $\RFib_0(\sO^{p/B}_\C) \simeq \C_{/B}$. 
\end{proof}

We now want to use this Reedy right fibration to construct a Segal object. For that we need to prove that $\sO_\C^{p/B}$ is a Segal Cartesian fibration such that $\RFib_0(\sO_\C^{p/B})$ and $\RFib_1(\sO_\C^{p/B})$ are representable. 

\begin{lemone} \label{lemma:main theorem fib Segal}
	$\sO_\C^{p/B}$ is a Segal Cartesian fibration.
\end{lemone}

\begin{proof}
	By assumption $\sO_\C$ is a Cartesian fibration and so also a Segal Cartesian fibration. 
	Moreover, according to \cref{lemma:main theorem constant fib}, $\Cech(\C_{/B})$ and $\Cech(\RFib_0(\sO_\C))$ are Segal Cartesian fibrations.
	However, being a Segal Cartesian fibration is a limit condition. Hence, using the fact that limits commute with pullbacks implies that their pullback, $\sO_\C^{p/B}$, is also a Segal Cartesian fibration.
\end{proof} 

\begin{lemone} \label{lemma:main theorem completeness}
	$\sO_\C^{p/B}$ is a Cartesian fibration if and only if $\C_{/B} \to \Oall_\C$ is an inclusion. 
\end{lemone}

\begin{proof}
	First of all $\sO_\C$ is a Cartesian fibration, hence $(\sO_\C)_{hoequiv} \simeq \Oall_\C$. Moreover, by \cref{lemma:main theorem constant fib}, 
	$\Cech(\C_{/B})_{hoequiv} = \C_{/B} \times_\C \C_{/B}$ and $\Cech(\Oall_\C) = \Oall_\C \times_\C \Oall_\C$. Hence, we have following pullback square of right fibrations of weak equivalences defined in \ref{eq:r hoequiv}
	\begin{center}
		\begin{tikzcd}
			(\sO^{p/B}_\C)_{hoequiv} \arrow[r] \arrow[d] & \Oall_\C \arrow[d, "\Delta"] \\
			\C_{/B} \times_\C \C_{/B} \arrow[r] & \Oall_\C \times_\C \Oall_\C 
		\end{tikzcd}
	\end{center} 
    using the fact that pullbacks commute with pullbacks. 
    This means that $\sO^{p/B}_\C$ is complete if and only if the following is a pullback square of right fibrations
	\begin{center}
		\begin{tikzcd}
			\C_{/B} \arrow[r] \arrow[d, "\Delta"] & \Oall_\C \arrow[d, "\Delta"] \\
			\C_{/B} \times_\C \C_{/B} \arrow[r] & \Oall_\C \times_\C \Oall_\C 
		\end{tikzcd}
	\end{center}
	which is equivalent to $\C_{/B} \to \Oall_\C$ being an inclusion (\fibref{item:over cat trunc}).
\end{proof}

 Now, we need to confirm the representability results. The representability of the $0$-level is immediate and so we will focus on the next level , meaning we want to prove that $\RFib_1(\sO_\C^{p/B})$ is representable. Before we do so, we need a better understanding of the objects and morphisms in the $\infty$-category $\RFib_1(\sO_\C^{p/B})$. Based on \cref{lemma:main theorem fib exists}, $\RFib_1(\sO_\C^{p/B})$ is defined as the following pullback of $\infty$-categories 
 \begin{equation} \label{eq:pb}
 	\begin{tikzcd}
 		\RFib_1(\sO_\C^{p/B}) \arrow[r] \arrow[d] & \RFib_1(\sO_\C) \arrow[d] \\
 		\C_{/B} \times_\C \C_{/B} \arrow[r] & \Oall_\C \times_\C\Oall_\C 
 	\end{tikzcd}
 \end{equation}
  Hence, we can explicitly describe an object in $\RFib_1(\sO_\C^{p/B})$ as a diagram of the form 
  \begin{equation} \label{eq:obj opb}
  	\begin{tikzcd}
  		E \arrow[d, "p"] & f^*E \arrow[r, "h"] \arrow[d, "f^*p"] \arrow[l] \arrow[dl, phantom, "\urcorner", very near start] & g^*E \arrow[d, "g^*p"] \arrow[r] \arrow[dr, phantom, "\ulcorner", very near start] & E \arrow[d, "p"] \\ 
  		B & X \arrow[l, "f"'] \arrow[r, "\id_X"] & X \arrow[r, "g"] & B 
    \end{tikzcd}
  \end{equation}
  with morphisms given by pullbacks of such diagrams. We denote such an object by the tuple $(X,f,g,h)$. A morphism $(X,f,g,h)$ to $(Y,i,j,k)$ correspond to diagrams of the form
 \begin{equation}\label{eq:mor opb}
 	\begin{tikzcd}
 		& E \arrow[dd, "p" near start] \arrow[dl, "\id_E"'] & & f^*E \arrow[rr, "h"] \arrow[dd, "f^*p" near start] \arrow[ll] \arrow[dl] \arrow[dl, phantom, "\urcorner", very near start] & & g^*E \arrow[dd, "g^*p" near start] \arrow[rr] \arrow[dl]  \arrow[dr, phantom, "\ulcorner", very near start] & & E \arrow[dd, "p" near start] \arrow[dl, "\id_E"]\\
 		E \arrow[dd, "p"] & & i^*E \arrow[rr, "k" near end] \arrow[dd, "i^*p" near start] \arrow[ll] \arrow[dl, phantom, "\urcorner", very near start] & & j^*E \arrow[dd, "j^*p" near start] \arrow[rr]  \arrow[dr, phantom, "\ulcorner", very near start] & & E \arrow[dd, "p" near start] \\
 		& B\arrow[dl, "\id_B"' near start] & & X \arrow[ll, "f"' near end] \arrow[rr, "\id_X" near start] \arrow[dl, "p"]& & X \arrow[rr, "g" near start]\arrow[dl, "p"] & & B \arrow[dl, "\id_B"]\\ 
 		B & & Y \arrow[ll, "i"'] \arrow[rr, "\id_Y"] & & Y \arrow[rr, "j"] & & B 
 	\end{tikzcd},
 \end{equation}
  where the vertical squares above $p$ and $\id_B$ are pullback diagrams. The source map (target map) is given by restricting to the back diagram (front diagram).
  
 \begin{lemone} \label{lemma:main theorem one rep}
 	The right fibration $\RFib_1(\sO_\C^{p/B})$ is represented by $\morphism{(\pi_1)^*E}{(\pi_2)^*E}_{B \times B}$.
 \end{lemone}

 \begin{proof}
 	We will prove that the $\infty$-category $\RFib_1(\sO_\C^{p/B})$ has a final object. The result then follows from \fibref{item:rep right fib final}. Let us first construct the object we want to prove is the final object, meaning we will construct a diagram of the form given in \ref{eq:obj opb}. Let $\pi_1,\pi_2: B \times B \to B$ and define $\cM = \morphism{(\pi_1)^*E}{(\pi_2)^*E}_{B \times B}$. By construction there are morphisms $(s,t): \cM \to B \times B$. We want to construct a morphism $s^*E \to t^*E$ over $\cM$, which requires us first computing these pullbacks. We have 
 	$$s^*E = E \times_B \morphism{E \times B}{B \times E}_{B \times B} \simeq (E \times B) \times_{B \times B} \morphism{E \times B}{B \times E}_{B \times B}$$
 	$$t^*E = E \times_B \morphism{E \times B}{B \times E}_{B \times B} \simeq (B \times E) \times_{B \times B} \morphism{E \times B}{B \times E}_{B \times B}$$
 	We can now define $\comp: s^*E \to t^*E$ by 
 	$$(c,\pi_2): (E \times B) \times_{B \times B} \morphism{E \times B}{B \times E}_{B \times B} \to (B \times E) \times_{B \times B} \morphism{E \times B}{B \times E}_{B \times B}$$
 	where $c: (E \times B) \times_{B \times B} \morphism{E \times B}{B \times E}_{B \times B} \to B \times E$ is the counit map of the adjunction.
 	Hence, we have constructed the object $(\cM,s,t,\comp)$
 	\begin{equation} \label{eq:final object} 
 		\begin{tikzcd}
 			E \arrow[d, "p"] & s^*E \arrow[r, "\comp"] \arrow[d, "s^*p"] \arrow[l]  \arrow[dl, phantom, "\urcorner", very near start] & t^*E \arrow[d, "t^*p"] \arrow[r]  \arrow[dr, phantom, "\ulcorner", very near start] & E \arrow[d, "p"] \\ 
 			B & \cM \arrow[l, "f"'] \arrow[r, "\id_X"] & \cM \arrow[r, "g"] & B 
 		\end{tikzcd}.
 	\end{equation}
    We will now prove it is final, by proving there is a (homotopically) unique pullback square from an arbitrary object $(X,f,g,h)$ depicted in \ref{eq:obj opb} to the object $(\cM,s,t,\comp)$ given in \ref{eq:final object}.
    
    Let $\cdot \leftarrow \cdot \rightarrow \cdot$ denote the span $\infty$-category with $3$ objects and $2$ morphisms.
    By the universal property of pullbacks we have the equivalence 
    $$\RFib_1((\sO^{p/B})_{/(\cM,s,t,\comp)})^{\simeq} \overset{\simeq}{\longrightarrow} ((\C^{\cdot \leftarrow \cdot \rightarrow \cdot})_{/B \leftarrow \cM \rightarrow B})^{\simeq},$$
    which is given by restricting the diagram \ref{eq:mor opb} to its bottom part.
    Moreover, we have the equivalence 
    $$((\C^{\cdot \leftarrow \cdot \rightarrow \cdot})_{/B \leftarrow \cM \rightarrow B})^{\simeq} \to (\C_{/\cM})^\simeq$$
    as $\cM$ is the limit of the diagram $B \leftarrow \cM \rightarrow B$ (\fibref{item:rep right fib limit}). Note, the inverse of this equivalence 
    $$(\C_{/\cM})^\simeq \to \RFib_1(\sO^{p/B}_{/(\cM,s,t,\comp)})^{\simeq}$$
    is given by pulling back the diagram $(\cM,s,t,\comp)$ along the map $p: X \to \cM$.
    Hence, the composition map 
    $$\Pb: (\C_{/\cM})^\simeq \to \RFib_1(\sO^{p/B})^{\simeq}$$
    takes a morphism $p: X \to \cM$ to the diagram $(X,p^*s,p^*t,p^*\comp)$. We want to prove this composition is an equivalence. 
    First of all we have a commutative triangle
    \begin{center}
    	\begin{tikzcd}
    		(\C_{/\cM})^\simeq \arrow[rr] \arrow[dr] & &  \RFib_1(\sO^{p/B})^{\simeq} \arrow[dl] \\
    		& (\C_{/B\times B})^\simeq
    	\end{tikzcd}.
    \end{center}
    Hence, it suffices to prove the fiber $\Pb_{(f,g)}$ is an equivalence of fibers over an object $(f,g): X \to B \times B$ in $(\C_{/B\times B})^\simeq$. By definition the fiber of the left hand map is $\Map_{/B \times B}((f,g),(s,t))$. Moreover, given the pullback square \ref{eq:pb}, the fiber of the right hand side is also the fiber of $\RFib_1(\sO_\C) \to \Oall_\C \times_\C \Oall_\C$ over $(f,g)$, which is precisely $\Map_{/X}(f^*E,g^*E)$. We now have the maps 
    \begin{align*}
    	\Map_{/B \times B}(X, \morphism{E \times B}{B \times E}_{B \times B}) & \to \Map_{/B \times B}(E \times B \times_{B\times B} X, E \times B \times_{B\times B} \morphism{E \times B}{B \times E}_{B \times B}) \\ 
    	& \to \Map_{/B\times B}(E \times B \times_{B\times B} X,B \times E),
    \end{align*}
    where the first map pulls back along $E \times B \to B \times B$ and the second post-composes with the counit map $E \times B \times_{B\times B} \morphism{E \times B}{B \times E}_{B \times B} \to B \times E$ and the composition is an equivalence as we have an adjunction.
    
    Next, we notice that the map $E \times B \times_{B\times B} X \to B \times B$ is given via post-composing the projection $E \times B \times_{B\times B} X \to X$ with $(f,g): X \to B \times B$ have the equivalence 
    $$ \Map_{/B\times B}(E \times B \times_{B\times B} X,B \times E) \to \Map_{/X}(E \times B \times_{B \times B}X, X \times_{B \times B} B \times E)$$
    given by pulling back along $X$. 
    
    Notice we have the two pullback squares 
    \begin{center}
    	\begin{tikzcd}
    	 X \times_{B \times B}E \times B \arrow[r] \arrow[d] & E \times B \arrow[r] \arrow[d, "p \times \id"] & E \arrow[d, "p"] \\
    	 X \arrow[r, "(f \comma g)"] \arrow[rr, "f"', bend right=20]& B \times B \arrow[r, "\pi_1"] & B 
    	\end{tikzcd}
    \end{center}
	which paste to one pullback diagram with bottom map $\pi_1(f,g) =f$ and so we have $f^*E \simeq X \times_{B \times B}E \times B $. Similarly we have $g^*E \simeq X \times_{B \times B}B \times E$. Hence we get an equivalence  
	$$\Map_{/X}(E \times B \times_{B \times B}X, X \times_{B \times B} B \times E) \to \Map_{/X}(f^*E,g^*E) .$$
	Combining the steps, we constructed the desired equivalence 
	$$\Pb_{(f,g)}:\Map_{/B \times B}((f,g),(s,t)) \to \Map_{/X}(f^*E,g^*E)$$
	explicitly constructed via pulling back the diagram $(\cM,s,t,\comp)$ via the map $X \to \cM$ proving that the map $\Pb$ is an equivalence.
	
    Finally, letting $\mathbbe{1}$ denote the final $\infty$-category, we have following pullback diagram of $\infty$-categories 
    \begin{center}
    	\begin{tikzcd}
    		\Map_{\RFib_1(\sO^{p/B})}((X,f,g,h),(\cM,s,t,\comp)) \arrow[r] \arrow[d] & \RFib_1(\sO^{p/B}_{/(\cM,s,t,\comp)})^{\simeq} \arrow[d, "\simeq"] \\
    		\mathbbe{1} \arrow[r ,"(X \comma f \comma g \comma h)"] & \RFib_1(\sO^{p/B})^{\simeq}
    	\end{tikzcd}.
    \end{center}
	We showed above that the right vertical map is an equivalence, hence the left hand vertical map is an equivalence as well, which proves the mapping space is contractible and so $(\cM,s,t,\comp)$ is final.
 \end{proof}

\begin{remone} \label{rem:like strict}
	The steps of the proof should remind us of the construction we made in \cref{sec:strict cat}. In other words, the whole setup of representable Cartesian fibrations allowed us to successively strictify the homotopy coherence, until we reduced it to the construction of one object which is the one object that also played a key role in the case of strict categories, as witnessed in \ref{eq:m}.
\end{remone}

\section{Univalence and Topos Theory} \label{sec:univ and topos}
 In this section we want try to classify all univalent morphisms in a locally Cartesian closed $\infty$-category and observe how our desire to classify them consistently moves us from  classical categories to higher categories and from categories to toposes. 
 Let us start by using the fact that $\infty$-categories have truncation to break down the question by truncation level. Truncations have been studied extensively in presentable $\infty$-categories \cite{lurie2009htt,rezk2010toposes,abfj2017blakersmassey,schlankyanovski2019eckmannhilton}, however, as we are working with finitely complete $\infty$-categories we will review one important result we need.
 
 \begin{lemone} \label{lemma:ntrunc}
 	Let $\C$ be a finitely complete $\infty$-category.
 	Let $p: E \to B$ be a morphism such that $B$ is $n$-truncated and $p$ is $k$-truncated, where $k \leq n$. Then $E$ is $n$-truncated as well. In particular, if $n > -2$, then a subject of an $n$-truncated object is also $n$-truncated
 \end{lemone}
 \begin{proof}
 	By \fibref{item:rep right fib limit}, the Yoneda lemma reflects limits. As $n$-truncated objects are characterized via a limit condition (\fibref{item:truncated}) it suffices to prove that if for every object $Z$, the map of spaces $\Map(Z,p): \Map(Z,E) \to \Map(Z,B)$ is $k$-truncated with $\Map(Z,B)$ $n$-truncated and $k \leq n$, then $\Map(Z,E)$ is $n$-truncated. However, this is a classical result about the category of spaces and truncations, which can, for example, be found in \cite[Lemma 8.6]{rezk2010toposes}
 \end{proof}
 With this lemma at hand, we can now prove following proposition, which was also proven in \cite[Lemma 2.12]{vergura2019localization} in the particular case of presentable $\infty$-categories.
 
 \begin{propone}\label{prop:trunc univ}
 	Let $\C$ be a finitely complete $\infty$-category and $p: E \to B$ a morphism such that pulling back along $E \times B \to B \times B$ has a right adjoint and let $n>-2$. If $p$ is $n$-truncated, then $B$ and $E$ are $(n+1)$-truncated.
 \end{propone}

\begin{proof}
	If $B$ is $(n+1)$-truncated, then $E$ is also $(n+1)$-truncated. Indeed, $p$ is $n$-truncated by assumption and so it follows from \cref{lemma:ntrunc}. Hence, it suffices to prove that $B$ is $(n+1)$-truncated, which is equivalent to proving that $\Delta: B \to B \times B$ is $n$-truncated (\fibref{item:truncated}). As $p$ is univalent, this is equivalent to $\n(p)_{hoequiv} \to B \times B$ being $n$-truncated (\cref{the:main theorem}).
	
	Now, we have a mono map $\n(p)_{hoequiv} \to \n(p)_1 = \morphism{E \times B}{B \times E}_{B \times B}$, and so, by \cref{lemma:ntrunc}, it suffices to prove that $\n(p)_1$ is $n$-truncated. This follows immediately from the fact that $p$ is $n$-truncated. Indeed $\id_B\times p: B \times E \to B \times B$ is $n$-truncated, as $n$-truncated morphisms are closed under products, and so the right hand side of the equivalence 
	$$\Map_{/B \times B}(X, \n(p)_1) \simeq \Map_{/B \times B}( X \times_{B \times B} E \times B, B \times E)$$
	is $n$-truncated, giving us the desired result.
\end{proof}
 
 The previous result leaves the case $n=-2$, for which we can in fact prove an even stronger statement.
 
 \begin{propone} \label{prop:neg one trunc univ}
 	Let $p: E \to B$ be $(-2)$-truncated. Then $p$ is univalent if and only if $B$ is $(-1)$-truncated if and only if $E$ is $(-1)$-truncated.
 \end{propone}

\begin{proof}
	As $p$ is $(-2)$-truncated, meaning an equivalence, $B$ is $(-1)$-truncated if and only if $E$ is $(-1)$-truncated. Hence it suffices to prove that $B$ is $(-1)$-truncated if and only $p$ is univalent.
	As $p:E \to B$ is an equivalence, $E \times B \to B \times B$ and $B \times E \to B \times B$ are also equivalences and $\morphism{E \times B}{B \times E}_{B \times B} \simeq B \times B$, where the source and target map are given by projection $\pi_1,\pi_2: B \times B \to B$. Hence, the Segal object $\n(p)$ is characterized level-wise by $\n(p)_n \simeq B^{n+1}$.
	
	Now, the map $d_1d_0: \n(p)_3 \simeq B^4 \to B^2 \simeq \n(p)_1$ restricts to the first and third coordinate and similarly, $d_1d_3$ restricts to the second and fourth coordinate and so $(d_1d_0,d_1d_3): \n(p)_3 \to \n(p)_1 \times \n(p)_1$ is an equivalence. Hence $\n(p)_{hoequiv} \simeq B \times B$, as it is given by pulling back $(d_1d_0,d_1d_3)$ along $B \times B \to B^4$. We conclude that $p$ is univalent if and only if $\Delta: B \to B \times B$ is an equivalence which is equivalent to $B$ being $(-1)$-truncated.
\end{proof}
 
 The previous proposition applies in particular to the identity map (which is indeed $(-2)$-truncated). Hence the identity map $\id_B$ is univalent if and only if $B$ is $(-1)$-truncated. Note this should be seen as a direct generalization of \cite[Examples 3.2.11]{kapulkinlumsdaine2012kanunivalent}, where they state that the only univalent identity maps come from the point and the empty set, which corresponds to the fact that these are the only $(-1)$-truncated spaces. 

We move on to the next level case. When is a $(-1)$-truncated map $p:E\hookrightarrow B$ univalent? By \cref{prop:trunc univ}, $E$ and $B$ need to be $0$-truncated. Hence, for every object $X$, $\Map_\C(X,B)$ is (equivalent to) a set and category structure on $\Map(X,B)$ reduces to a preorder, meaning the functor represented by the Segal object $\Map_\C(-,\n(p)): \C^{op} \to\cat_\infty$ takes value in preorders. 
We can very explicitly describe the preorder structure on the set $\Map(X,B)$ as follows: For two elements $f,g: X \to B$ we have $f \leq g$ if and only if $f^*E$ is subobject of $g^*E$ over $X$, meaning there is a morphism $f^*E \to g^*E$ over $X$, which will be (homotopically) unique, if it exists.

Two objects $f,g$ in the preorder $\Map(X,B)$ are isomorphic if and only if $f \leq g$ and $g \leq f$. We conclude that $p$ is univalent if and only if the internal preorder structure on $B$ is complete, meaning for every object $X$ and morphisms $f,g: X \to B$  we have 
\begin{equation} \label{eq:univ neg one}
f = g \text{ in } \Map_\C(X,B) \Leftrightarrow f^*E \simeq g^*E \text{ over } X.
\end{equation}
This generalizes \cref{cor:complete zero cat} from $1$-categories to $\infty$-categories, meaning univalence is equivalent to the preorder structure on $\Map_\C(X,B)$ being a poset. 

Using our intuition from \cref{sec:univ segal obj} we can reformulate this statement in terms of functors. For an object $X$, let $\Sub_\C(X)$ be the set of isomorphism classes of subobjects of $X$ ($(-1)$-truncated morphisms with target $X$). As $\C$ has pullbacks, this gives us a functor  $\Sub_\C: \C^{op} \to \set$. The inclusion $p:E \to B$ induces a natural transformation 
\begin{equation} \label{eq:univ neg one functor}
\Map_\C(-,B) \to \Sub_\C
\end{equation}
which takes a morphism $f:X \to B$ to the pullback $f^*E \to X$.
Using $\Sub$ the statement in \ref{eq:univ neg one} is now equivalent to stating the natural transformation is an injection. 
We would like to internalize such statements. However, that is generally not possible as the functor $\Sub$ is not representable in a general locally Cartesian closed category. Hence, we need to introduce the following definition.

\begin{defone} \label{def:soc}
	Let $\C$ be a finitely complete $\infty$-category.
	A {\it subobject classifier} $\Omega$ is an object representing the functor $\Sub: \C^{op} \to \set$, meaning there is a natural isomorphism of set-valued functors $\Map_\C(-,\Omega) \cong \Sub$.
\end{defone}

Concretely, the bijection $\Map_\C(\Omega, \Omega) \cong \Sub(\Omega)$ implies that there is a {\it universal subobject} $t: 1 \to \Omega$ (corresponding to $\id_\Omega$) and all other $(-1)$-truncated morphisms can be obtained as its pullback.
So, let us now assume $\C$ has a subobject classifier, then the statement in \ref{eq:univ neg one functor} is equivalent to $\Map_\C(-,B) \to \Map_\C(-,\Omega)$ is an injective map of sets, which (by the Yoneda lemma \fibref{item:over cat trunc}) is equivalent to $B \to \Omega$ being mono. Hence, a $(-1)$-truncated morphism $p: E \to B$ is univalent if and only if the corresponding map $B \to \Omega$ is mono.
In particular, the universal subobject $t:1 \to \Omega$ must also be univalent. 

While trying to classify all $(-1)$-truncated univalent morphisms, we have stumbled upon locally Cartesian closed $\infty$-categories with subobject classier. In the context of classical category theory such categories have been studied extensively as a foundation for mathematics and are known as {\it elementary toposes} \cite{tierney1973elementarytopos,johnstone2002elephanti,johnstone2002elephantsii,maclanemoerdijk1994topos}. In particular, elementary toposes include the category of sets and categories of set-valued presheaves. Hence we have just made following valuable observation about elementary toposes.

\begin{propone} \label{prop:mono univ eht}
	Let $\C$ be a locally Cartesian $\infty$-category with subobject classifier. Then a $(-1)$-truncated map is univalent if and only if its classification map to the subobject classifier is $(-1)$-truncated. This in particular applies to all elementary $1$-toposes.
\end{propone}

We now want to move on to $0$-truncated univalent maps. Given how elementary $1$-toposes naturally arose out of our effort to understand $(-1)$-truncated univalent maps, we will start there. Hence, we want to classify all $0$-truncated univalent morphisms in an elementary $1$-topos. Our first hope might be that there are no $0$-truncated univalent morphisms and we are already done. In fact here is a promising example.

\begin{lemone} \label{lemma:univ set}
	If a map in $\set$ is univalent then it has to be mono.
\end{lemone}

\begin{proof}
	If the map $p: E \to B$ is univalent in $\set$ then for every element $b \in B$ then fiber map $E_b \to \{ b \}$ also has to be 
	univalent as it is the pullback along the mono map $b: \{b\} \xrightarrow{\subseteq} B$, by \cref{prop:univ if mono}) . However, by \cref{ex:univalent over point} this only happens if
	$E_b$ has no non-trivial automorphisms. In the category of sets this only holds if $E_b$ is either empty or has one object.
	We just showed that the fiber over each point $b$ is either empty or has one point and so $E \to B$ is an injection of sets.
\end{proof}

This result suggests another way to think of univalent morphisms in $\set$. For a given map of sets $p:T \to S$, we can construct a functor $\Fib(p): S \to \set$, which takes an element $s \in S$ to the fiber $p^{-1}(s)$. Then $p$ is univalent if and only if $\Fib(p): S \to \set$ is fully faithful.
Indeed, $p$ is univalent if we have 
$$\Hom_{\set}(\Fib(p)(s),\Fib(p)(s')) = 
\begin{cases}
	\{ \id\} & \text{ if } s= s' \\
	\emptyset & \text{ if } s \neq s' \\ 
\end{cases}
$$
which is precisely the condition that a functor out of $S$ (which as only identity maps) is fully faithful. 
This result should remind us of \cref{prop:mono univ eht}, where now $\set$ is trying to play the role of the subobject classifier. However, there are two main differences that complicate the picture: First, $\set$ is a category and so not really an object in the category. Second, even if we just took the objects, then the collection would be too large to be an element in the category.

We can combine \cref{lemma:univ set} and \cref{prop:mono univ eht} to completely characterize the poset $\Univ_{\set}$. 
The subobject classifier in sets is the set $\{0,1\}$ and so the poset of univalent morphisms in set consists of the following morphisms:
\begin{center}
	$\Bigg\{$
	\begin{tabular}{ccccccc}
		$\emptyset$ & & $\emptyset$ & & $\{ 1 \}$ & & $\{ 1 \}$ \\
		$\downarrow$ & , & $\downarrow$ & , & $\downarrow$ & , & $\downarrow$ \\
		$\emptyset$ & & $\{ 1 \}$ & & $\{ 1 \}$ & & $\{ 0, 1 \}$
	\end{tabular}
    $\Bigg\}$
\end{center}
Here the poset structure comes from the subsets of $\{0,1\}$.

However, looking at a slightly more complicated example reveals this does not generalize.

\begin{exone} \label{ex:sthree univ}
    Let $G = S_3$ the group of permutation of $3$ elements.
	The category of $G$-sets is an elementary topos as it is a category of set-valued presheaves.
	Let $S = \{ 1,2,3 \}$, which comes with an obvious $G$ action.
	Then $S$ has no nontrivial automorphisms. Indeed, if $\sigma: S \to S$ is an automorphism
	then for every element $\tau$ in $G$ we need to have $\sigma \tau = \tau \sigma$ in order to satisfy the equivariance condition.
	However, this is only satisfied by the identity as $S_3$ has a trivial center.
	\par 
	Thus $S$ has no non-trivial automorphism in the category of $G$-sets. By \cref{ex:univalent over point} we deduce that the map 
	$S \to *$ is univalent. However, this map is not mono as a mono map in the category of presheaves is an injection of the underlying sets.
\end{exone}

The example illustrates that we do need to worry about $0$-truncated univalent morphisms in elementary toposes. Hence, our next aim is to classify them via an object similar to the subobject classifier. However, we run into the same two problems we encountered when trying to find an ``object classifier" in the category of sets. By \cref{prop:trunc univ}, if $p:E \to B$ is $0$-truncated and univalent, then $B$ is $1$-truncated. Hence, repeating the steps above we need to analyze the natural transformation $\Map(-,B) \to (\C_{/-})^\simeq$ similar to \ref{eq:univ neg one functor}. However, $(\C_{/-})^\simeq$ is a groupoid and hence can only be represented by a set if all isomorphisms are trivial, meaning we cannot give a definition analogous to \cref{def:soc}. The only possible way is to increase the categorical level. This naturally leads us to the following question:

\begin{queone}
	Given an elementary $1$-topos $\E$ does there exists an $\infty$-category $\hat{\E}$ such that $\E$ is equivalent to its subcategory of $0$-truncated objects, $\tau_0\hat{\E}$, and in which all univalent morphisms can be effectively characterized?
\end{queone}
This question has not been answered in full generality and is in fact only expected to hold in some (important) cases. We now want to focus on a case where this question has in fact an affirmative answer. 
Assume that the elementary $1$-topos $\E$ is additionally a {\it locally presentable category}. Such toposes are usually called {\it Grothendieck $1$-toposes} and were developed by Grothendieck and Bourbaki in the study of algebraic geometry \cite{sga1972tome1,sga1972tome2,sga1972tome3}. Grothendieck toposes can be characterized alternatively as {\it left exact localizations of presheaf categories}: A category $\G$ is a Grothendieck topos if and only if there exists an adjunction 
\begin{equation} \label{eq:topos}
 \begin{tikzcd}
 	\Fun(\C^{op},\set) \arrow[r, "a", shift left=1.8, "\bot"'] & \G \arrow[l, "i", shift left=1.8]
 \end{tikzcd}
\end{equation}
where $\C$ is small, $i$ is fully faithful and $a$ preserves finite limits \cite[Section V.4]{maclanemoerdijk1994topos}. The adjunction \ref{eq:topos} can be directly generalized to the $\infty$-categorical setting. An $\infty$-category $\G$ is a Grothendieck $\infty$-topos if there exists an adjunction 
\begin{equation} \label{eq:infty topos}
	\begin{tikzcd}
		\Fun(\C^{op},\s) \arrow[r, "a", shift left=1.8, "\bot"'] & \G \arrow[l, "i", shift left=1.8]
	\end{tikzcd}
\end{equation}
where $\C$ is small, $i$ is fully faithful and $a$ is accessible and preserves finite limits \cite{lurie2009htt,rezk2010toposes}.
 Every Grothendieck $1$-topos $\G$ has a (non-unique) {\it enveloping $\infty$-topos} \cite[Section 11]{rezk2010toposes}, meaning there exists a Grothendieck $\infty$-topos $\hat{\G}$, such that $\tau_0\hat{\G} \simeq \G$. So, we have successfully lifted our Grothendieck topos to a Grothendieck $\infty$-topos. Now, can we characterize univalent morphisms effectively in $\hat{\G}$?
 
 Before we proceed let us make a detour and do an extended analysis of the simplest possible case of a Grothendieck $\infty$-topos, namely when we take $\C$ to be the final $\infty$-category and $a = i = \id$ in \ref{eq:infty topos}, which gives us $\G = \s$. How can we completely characterize univalent morphisms in $\s$?
 
 Assume $p: Y \to X$ is a univalent. We can assume $X$ is connected, as by \cref{prop:univ if mono}, $p$ restricted to each path-component of $X$ needs to be univalent as well. Fix a point $x$ in $X$ and denote the fiber of $p$ over $x$ by $Y_x$. Then univalence states that $\Eq(Y_x) \simeq \Omega X$, which means that $X \simeq B(\Eq(Y_x))$. The fact that the pullback of $Y \to B(\Eq(Y_x))$ along the point is $\Eq(Y_x) = \Omega B \Eq(Y_x)$ implies that $Y$ must be contractible. We conclude that a univalent morphism of spaces with connected codomain is a classifying map of the form $E(\Eq(X)) \to B(\Eq(X))$ for some space $X$. This observation can in fact be generalized and used to characterize univalent morphism with connected codomain in arbitrary Grothendieck $\infty$-toposes \cite[Subsection 6.1]{gepnerkock2017univalence}
 
 We can use this result to now characterize all univalent morphisms of spaces. If an arbitrary map $p: E \to B$ is univalent, then we necessarily have $p \simeq \coprod_{i \in I} p_{X_i}$, where $p_{X_i}: E(\Eq(X_i)) \to B(\Eq(X_i))$ is a universal map. However, this is not sufficient as for $i \neq j$ we could have $X_i \simeq X_j$, which would then imply $p_{X_i} \simeq p_{X_j}$ violating univalence. Indeed, the map from the final object $* \to X_i$, $* \to X_j$ would not be equivalent, whereas the fibers $X_i$ and $X_j$ would be equivalent. Hence, an arbitrary univalent morphism in $\s$ is of the form 
 \begin{equation} \label{eq:univ spaces}
 \ds\coprod_{i \in I}(p_i:E(\Eq(X_i)) \to B(\Eq(X_i)))
 \end{equation} 
 with $X_i \not\simeq X_j$, for $ i \neq j$.
 
 We can use this construct some very useful univalent morphisms. Let $\kappa$ be some cardinal and let $(\s^\kappa)^\simeq$ be the maximal sub-$\infty$-groupoid of the full sub-category of $\kappa$-compact spaces. Moreover, let $(\s_*^\kappa)^\simeq$ be the maximal sub-$\infty$-groupoid of the full sub-category of $\kappa$-compact {\it pointed} spaces and notice it comes with a forgetful map $(\s_*^\kappa)^\simeq \to (\s^\kappa)^\simeq$.
 It satisfies precisely the condition stated above: each path component is equivalent to the space self-equivalences of a certain space, but no two different path components are equivalent. Moreover, the fiber over a given point representing an object $X$ is precisely the space $X$ itself. Hence $(\s_*^\kappa)^\simeq \to (\s^\kappa)^\simeq$ is univalent. More importantly, we can obtain every $\kappa$-compact morphism $p:Y \to X$ as a pullback of this map and so we have following result: A map of $\kappa$-compact spaces $Y \to X$ is univalent if and only if the corresponding map $X \to (\s^{\kappa})^\simeq$ is $(-1)$-truncated.

 We can in fact externalize this result using the straightening construction. Following the equivalence $\widehat{\s}_{/X} \simeq \Fun(X,\widehat{\s})$ (\fibref{item:grothendieck}/\fibref{item:rfib over space}), we can (homotopically uniquely) associate a functor $X \to \widehat{\s}$ to every map $p:Y \to X$. Then $p$ is univalent if and only if the associated map is fully faithful, meaning we have an equivalence 
 \begin{equation} \label{eq:univ ff}
  (\Univ_\s)_{/X} \simeq \Fun^{FF}(X,\widehat{\s})
 \end{equation}
 where the right hand side are the fully faithful functors out of $X$. This can be seen as an appropriate generalization of \cref{prop:mono univ eht} to the $\infty$-category of spaces.
 
 While the external characterization given in \ref{eq:univ ff} cannot be proven (or even stated) for every Grothendieck $\infty$-topos, the internal classification has been successfully generalized. In \cite[Theorem 6.1.6.8]{lurie2009htt} Lurie proves that if $\hat{\G}$ is a Grothendieck $\infty$-topos and $\kappa$ is a large enough cardinal, then there exists a {\it universal family} $\U^\kappa_* \to \U^\kappa$, such that the induced map of right fibration (via the Yoneda lemma \fibref{item:rep right fib}) $\Map_{\hat{\G}}(-,\U^\kappa) \to \Oall_{\hat{\G}}$ is an embedding and the essential image precisely corresponds to the {\it $\kappa$-compact objects}. 
 Presentability implies that every morphism is $\kappa$-small for some cardinal $\kappa$.
 We can now use these results to prove the following, which has also been similarly analyzed in \cite{gepnerkock2017univalence}.
 
\begin{propone} \label{prop:univ n topos}
	Let $\G$ be a Grothendieck topos with enveloping $\infty$-topos $\hat{\G}$. Moreover, let $p: E \to B$ be a morphism in $\G$. Then the following are equivalent:
	\begin{enumerate}
		\item $p$ is univalent in $\G$.
		\item $p$ is univalent in $\hat{\G}$.
		\item The classifying map $B \to \U^\kappa$ classifying $p$ is mono. 
	\end{enumerate}
\end{propone}

\begin{proof}
	$(1) \Leftrightarrow (2)$ Truncation is a locally Cartesian closed localization of locally Cartesian closed presentable $\infty$-categories \cite[Lemma 1.8]{gepnerkock2017univalence} and so the result follows from \cref{prop:lccc localization reflect univalence}.
	
	$(2) \Leftrightarrow (3)$ The morphism $p$ is $\kappa$-compact for some weakly inaccessible cardinal $\kappa$. Hence, there exists a pullback of the form
	\begin{center}
	 \begin{tikzcd}
	 	E \arrow[r] \arrow[d]  & \U^{\kappa}_* \arrow[d] \\
	 	B \arrow[r] & \U^{\kappa}
	 \end{tikzcd}	
    \end{center}
	The equivalent conditions now follow from \cref{prop:univ if mono}.
\end{proof}
  
  Up until now we gave an abstract result helping us characterize univalent morphisms as certain pullbacks. However, can we use these results to classify univalent morphisms, the way we have done, for example, in the category of sets (\cref{lemma:univ set})? Unfortunately, this is not the case and even the simplest cases beyond $(-1)$-truncated univalent morphisms cannot be easily classified. On the other side the additional complication points to interesting connections to other branches of mathematics.
     
   Let us try to characterize $1$-truncated univalent morphisms of spaces $E \to *$ where $E$ is connected and $*$ is contractible. In that case $E \simeq BG$ where $G$ is a discrete group. By \cref{ex:univalent over point}, $BG \to *$ is univalent if and only if $\Eq(BG)$ is contractible. We want to characterize $\Eq(BG)$ in terms of $G$. Recall, there is a {\it fundamental groupoid} functor \cite{dieck2008algebraictopology}, which gives us an equivalence between $1$-truncated connected spaces and connected groupoids
   $$\ds\Pi_{\leq 1}: \s_{\leq 1}^{cn} \to \Grpd^{cn}.$$
   Moreover, the {\it nerve functor} gives us equivalence $N: \Grp \to \Grpd^{cn}$ that takes a group to a groupoid with one object and mapping set that group. Let us fix an inverse $(-)^{grp}: \Grpd^{cn} \to \Grp$. Composing with the fundamental groupoid gives us an equivalence $\Pi_{\leq 1}^{grp}: \s_{\leq 1}^{cn} \to \Grp$. It takes a space $X$ to the fundamental group $\pi_1(X,x_0)$ (where the base point $x_0$ in $X$ is chosen by the functor $(-)^{grp}$), and a continuous map $f: X \to Y$ to a {\it conjugacy class} $\pi_1([f]):\pi_1(X,x_0) \to \pi_1(Y,y_0)$ (where we have to restrict to conjugacy classes as we have fixed a base point).
   Using this equivalence, we can associate to a space $X$ the group $\Pi_{\leq 1}^{cn}(\Eq(X))$, which can be explicitly described as conjugacy classes of automorphisms of $\pi_1(X,x_0)$. Hence, for univalence, we must have that all equivalences are in the same equivalence class, meaning they are conjugate to each other. However, we also need to know that no morphism is conjugate to itself in a non-trivial way (which would correspond to a loop in the space $\Eq(X)$). We want to give a more explicit description just in terms of groups, which requires us to review some group theory \cite{rotman1995groups,hungerford1980algebra,dummitfoote2004algebra}.
   
   For a group $G$, we denote the {\it center} (the normal subgroup that commutes with all elements) by $Z(G)$.
   Moreover, let $\Inn(G)$, the {\it inner automorphisms}, be the normal subgroup of $\Aut(G)$ consisting of conjugation automorphisms $\theta_g:  G\to G$, $\theta_g(h) = ghg^{-1}$. Define $\Out(G)$, the {\it outer automorphisms}, as the quotient group $\Aut(G)/\Inn(G)$. 
   The description in the previous paragraph means that we have an isomorphism of group $\Pi_{\leq 1}^{cn}(\Eq(X)) \cong \Out(G)$. So in order for $BG$ to be univalent $\Out(G)$ needs to be trivial. Moreover, for a given group isomorphism a conjugation action is trivial if the conjugating element is in the center of the group, meaning we need the center of the group to be trivial as well. A group that satisfies these two conditions is called {\it complete}.
   
   In fact we can phrase the completeness condition in a way that should look more familiar from a univalent perspective. We have following exact sequence of groups 
   \begin{equation}
   	\begin{tikzcd}
   		1 \arrow[r] & Z(G) \arrow[r, hookrightarrow] & G \arrow[r, "\theta"] & \Aut(G) \arrow[r, twoheadrightarrow] & \Out(G) \arrow[r]& 1
   	\end{tikzcd}
   \end{equation}
   meaning $G$ is complete if and only if $\theta: G \to \Aut(G)$ is a group isomorphism. Hence, $1$-truncated univalent maps over the final objects are as complicated as complete groups, which we can summarized in the following result.
   
   \begin{corone} \label{cor:univalent complete group}
   	There is a bijection between complete groups and connected $1$-truncated univalent spaces over the point.
   \end{corone}
   
   The study of complete groups is an ongoing project. We know that $S_n$ is complete for $n \neq 2,6$  \cite[Section 7]{rotman1995groups}. Moreover, there is a ``{\it completion procedure}" that constructs a complete group out of the chain 
   $$G \to \Aut(G) \to \Aut(\Aut(G)) \to ... $$ 
   which is described in more detail in \cite{wielandt1939complete,hamkins1998complete}. One interesting question would be to better understand the relation between the completion procedure for groups and the univalent completion discussed in \cref{sec:univ segal obj}, however, we will not address this question here.
  
  Interestingly enough univalence does not seem to have similar valuable implication in the category of groups. In fact we can characterize all univalent morphisms in in the category of groups, $\Grp$, directly using \cref{def:univalence}.
  
 \begin{lemone} \label{lemma:univ groups}
 	Let $\Grp$ be the $1$-category of groups. Then the poset $\Univ_{\Grp}$ is isomorphic to the following poset
 	\begin{center}
 		$\Bigg\{$
 		\begin{tabular}{ccccccc}
 			$1$ & & $1$ & & $\mathbb{Z}/2$ & & $\mathbb{Z}/2$ \\
 			$\downarrow$ & $\leq$ & $\downarrow $ & , & $\downarrow$ & $\leq$ & $\downarrow 0$ \\
 			$1$ & & $\mathbb{Z}/2$ & & $1$ & & $\mathbb{Z}/2$
 		\end{tabular}
 	$\Bigg\}$
 	\end{center}.
 \end{lemone}
 \begin{proof} 
 	 The proof breaks down into several steps:
 	 
 	 {\it (1) Univalent morphism over trivial groups:}
 	 By \cref{ex:univalent over point} the univalent groups over the trivial group are the ones without non-trivial automorphisms. In this step we show that the only groups with trivial automorphism groups are the trivial group $1$ and $\mathbb{Z}/2$. Indeed, if the group is not abelian then conjugation by an element must be a non-trivial automorphisms. However, if the group is abelian then the inverse map is an automorphism, which is the identity if and only if every element has order $2$. Moreover, if a group with all elements of order $2$ has more than one non-trivial element then flipping them gives us another automorphism. Hence, we are left with the trivial group and $\mathbb{Z}/2$.
 	
 	{\it (2) Univalent isomorphisms are trivial:}
 	By \cref{prop:neg one trunc univ}, an isomorphism is univalent if and only if the domain is a subobject of the trivial group $1$, which means it must be the trivial group. 
 	
 	{\it (3) Domains of univalent morphisms are $2$-groups:}
 	Next, we claim that if $\varphi: H \to G$ is univalent then every element in $G$ has order $1$ or $2$. Fix an element $g \in G$. Restriction along the subgroup generated by $g$, $<g>$ gives us a univalent morphism (\cref{prop:univ if mono}), which we also denote by $\varphi: H \to <g>$. By assumption the map is surjective. Moreover by Step (1) the kernel has to be trivial or $\mathbb{Z}/2$ as the map from the kernel to $1$ will be univalent. If the kernel is trivial the map is an isomorphism, which, by the Step (2), is univalent if and only if $<g>=1$, meaning $g$ is the trivial element. 
 	
 	Hence, let us assume the kernel is $\mathbb{Z}/2$. This means we have following short exact sequence of groups 
 	$$
 	1 \to \mathbb{Z}/2 \to H \to <g> \to 1
 	$$
 	where $<g>$ is cyclic of order prime or infinite. This extension only has trivial solutions. Indeed, $\mathbb{Z}/2$ is a normal subgroup of $H$ by assumption and index $2$ subgroups are always normal. Hence $H \cong \mathbb{Z}/2 \times <g>$. We now have following two pullback squares of groups (where $\iota$ denote the inversion isomorphism)
 	\begin{center}
 		\begin{tikzcd}[row sep=0.4in]
 			\mathbb{Z}/2 \times <g> \arrow[r, "\id"] \arrow[d, "\pi_2"] & \mathbb{Z}/2 \times <g> \arrow[d, "\pi_2"]\\
 			<g> \arrow[r, "\id"] & <g>
 		\end{tikzcd}
 	 \ , \
 	 	\begin{tikzcd}[row sep=0.4in]
 	 		\mathbb{Z}/2 \times <g> \arrow[r, "\id \times \iota"] \arrow[d, "\pi_2"] & \mathbb{Z}/2 \times <g> \arrow[d, "\pi_2"] \\
 	 		<g> \arrow[r, "\iota"] & <g>
 	 	\end{tikzcd}.
 	\end{center}
    The fact that the pullbacks are equal, and univalence, implies that $\id$ and $\iota$ must be equal, meaning $ g= g^{-1}$, proving that $g$ must have order $2$.
 	
 	{\it (4) Univalent morphisms over $\mathbb{Z}/2$:}
    Let us characterize all univalent morphisms over $\mathbb{Z}/2$. By Step (1), we know that the kernel has to be trivial or $\mathbb{Z}/2$. If the kernel is trivial then the map is either the identity $\mathbb{Z}/2 \to \mathbb{Z}/2$, which is not univalent as $\mathbb{Z}/2$ is not a subobject of the trivial group (by Step (2)), or $1 \to \mathbb{Z}/2$, which is in fact univalent. Indeed, a group homomorphism $\varphi: H \to G$ is a pullback of $1 \to \mathbb{Z}/2$ if and only if it is either an isomorphism, in which case it is the pullback of the constant map $ G\to \mathbb{Z}/2$, or $H$ is an index $2$ subgroup, in which case there is a unique morphism $G \to \mathbb{Z}/2$ that makes $\varphi$ a pullback of $1 \to \mathbb{Z}/2$.
    
    If the kernel is not trivial, then by Step (1) it is just $\mathbb{Z}/2$. Let us first assume it is also surjective. We now need to solve the extension problem
 	$$
 	1 \to \mathbb{Z}/2 \to H \to \mathbb{Z}/2 \to 1.
 	$$
 	The group has order $4$ and so must be either $\mathbb{Z}/4$ or $\mathbb{Z}/2 \times \mathbb{Z}/2$. The map $\mathbb{Z}/4 \to \mathbb{Z}/2$ has a non-trivial automorphism (flipping $1$ and $3$, but not changing $0$ and $2$) and so cannot be univalent. This leaves us with the projection $\pi_1: \mathbb{Z}/2 \times \mathbb{Z}/2 \to \mathbb{Z}/2$, which is also not univalent. Indeed, the pullbacks of $\pi_1$ along $0,\id: \mathbb{Z}/2 \to \mathbb{Z}/2$ are the same although the maps are not equal.
 	
 	On the other hand if the kernel is not trivial, but the map is not surjective, then we have the $0$ map $\mathbb{Z}/2 \to \mathbb{Z}/2$, which is univalent as well. Indeed a morphism $\varphi: H \to G$ is a pullback of $0$ if and only if there is an isomorphism $H \cong \mathrm{Im}(\varphi) \times \mathbb{Z}/2$, which can only hold in a unique way.
 	
 	{\it (5) Domain of univalent morphism must have at most two elements:}
 	Finally, assume that $H \to G$ is univalent with $G$ more than $2$ elements. Then there must be an injective map $\mathbb{Z}/2 \times \mathbb{Z}/2 \to G$ and so the restriction $H \to \mathbb{Z}/2 \times \mathbb{Z}/2$ is univalent as well (again by (\cref{prop:univ if mono}). We know have three injective maps $\mathbb{Z}/2 \to \mathbb{Z}/2 \times \mathbb{Z}/2$ and so the restriction to all three maps must be univalent. However, there are only two univalent morphisms over $\mathbb{Z}/2$. Hence two of the fibers are isomorphic, meaning the corresponding maps must be isomorphic as well. This proves that there cannot be two copies of $\mathbb{Z}/2$ in $G$.
 	
 	{\it (6) Combining the steps:}
 	We have proven in Step (5) that every univalent morphism must have either trivial domain or domain equal $\mathbb{Z}/2$. By Step (1) the ones with trivial domain must be $1 \to 1$ and $\mathbb{Z}/2 \to 1$. By Step (4), the univalent morphisms over $\mathbb{Z}/2$ must be of the form, $0: \mathbb{Z}/2 \to \mathbb{Z}/2$ and $1 \to \mathbb{Z}/2$. This finishes the proof. 
 \end{proof}
  
  Notice, \cref{lemma:univ groups} gives a justification for the assumptions in  \cref{lemma:univ leq p} and \cref{lemma:univ weak meets}. Indeed, the poset $\Univ_{\Grp}$ evidently does not have joins.
 
  We can summarize the work we have done in this section as follows: 
 \begin{itemize}
 	\item Trying to characterize $(-1)$-truncated univalent morphisms in the $\infty$-categorical setting encourages us to introduce subobject classifiers, moving us in the direction of elementary toposes.
 	\item Trying to characterize other univalent morphisms in the elementary topos setting forces us to consider a larger $\infty$-category.
 \end{itemize}
 Hence, the solution is to study $\infty$-categorical elementary topos theory!
 The presentable case was covered using Grothendieck $\infty$-toposes, however the non-presentable case needs further analysis and only first steps have been taken \cite{rasekh2018elementarytopos,stenzel2020comprehension}. 
 As we cannot give a definitive answer yet, we will leave the reader with two examples that point in two different directions.
 
 \begin{exone}
 	The category of finite sets is an elementary topos, which implies that the category of finite $G$-sets is also an elementary topos. We saw in \cref{ex:sthree univ} that the finite set $\{1,2,3\}$ with the action by $S_3$ is univalent over the final object. So we would hope there exists an $\infty$-category $\E$ such that $\tau_0\E$ is equivalent to finite $G$-sets. 
 	
 	However, this seems very unlikely. Indeed, any locally Cartesian closed $\infty$-category with subobject classifiers that satisfies some additional very mild assumptions (finite colimits and descent \cite[Section 6.1.3]{lurie2009htt}) will have a natural number object \cite{rasekh2021nno}, which does not exists in the category of finite $G$-sets.
 \end{exone}

  We might hope that we can at some point exploit the fact that finite sets are the full subcategory of the category of sets. Let us end on a positive note.
  
 \begin{exone}
 	Let $\E$ be an elementary $1$-topos. Then $\Hom_\E(1,\Omega)= \Sub(1)$ has the structure of a poset given by inclusion of subobjects and we can choose a filter \cite{changkeisler1990modeltheory,los1955ultraproduct} and construct a new elementary topos, called the {\it filter quotient} and denote $\E_\Phi$, which inherits some properties of the original topos, however, can also differ in important regards  \cite{adelmanjohnstone1982serreclasses}. Concretely, $\E_\Phi$ need not be presentable or cocomplete, even if $\E$ is \cite[Example A2.1.13]{johnstone2002elephanti}. Hence we cannot use \cref{prop:univ n topos} to characterize univalent morphisms. 
 	
 	Fortunately, in this case we can in fact lift $\E_\Phi$ to an $\infty$-category $\hat{\E}_\Phi$ \cite{rasekh2020filterquotient}. Let us focus on one example: the category of $I$-indexed sets $\E = \set^I$ where $I$ is a small set. In that case a filter on $\set^I$ is a filter of subsets of $I$ and the filter construction is called the {\it filter product} $\prod_\Phi \set$. Moreover, the $\infty$-categorical lift is given by the filter product on the $\infty$-category of $I$-indexed spaces $\s^I$, which is also denoted $\prod_\Phi \s$. Moreover, $\prod_\Phi \s$ has object classifiers given by the ones in $\s^I$ (which are just level-wise given by the ones in spaces analogous to \ref{eq:univ ff}).
 	
 	Although the filter products are not presentable $\infty$-categories they still have truncation functors \cite[Section 6]{rasekh2018truncations}.
 	Hence, applying \cref{prop:univ n topos}, we can hence deduce that a morphism in the $1$-category $\prod_\Phi \set$ is univalent if and only if it is a mono pullback of object classifiers in $\prod_\Phi \s$. 
 \end{exone}
 
 Finally note that although the study of univalence moves us towards topos theory, there are in fact presentable $\infty$-categories that are not toposes but where we do want to study and classify all univalent morphisms. One example is the $\infty$-category of motivic spaces, where some univalent morphisms have been characterized in \cite[Section 6.8]{gepnerkock2017univalence}.

\section{Further Examples} \label{sec:examples}
In this last section we want to look at univalence beyond the locally Cartesian closed settings: the $\infty$-category of $\infty$-categories and pointed $\infty$-categories. 

{\bf The $\infty$-Category of $\infty$-Categories:}
We want to study univalence in the very large $\infty$-category of large $\infty$-categories, $\widehat{\cat}_\infty$ (using the size conventions established in \fibref{item:universes}). Recall, by \fibref{item:inftycategories}, $\widehat{\cat}_\infty$  is the underlying quasi-category of $\infty$-categories, which are the objects of our chosen $\infty$-cosmos $\K$ (\fibref{item:cosmos}).

First observe that $\widehat{\cat}_\infty$ has all finite limits and so we can study the $\infty$-category $\Oall_{\widehat{\cat}_\infty} \to \widehat{\cat_\infty}$ and try to characterize the $(-1)$-truncated morphisms in that category. In particular we can immediately make following observation. 

\begin{exone}
	Let $\C$ be an $\infty$-category, then, following \cref{ex:univalent over point}, $\C$ over the final object is univalent if and only if the $\infty$-category of self-equivalences is trivial. This in particular includes categories such as $[0]$ or $[1]$.
\end{exone}

On the other hand, there is currently no clear way to characterize all univalent morphisms with non-trivial codomain in $\widehat{\cat}_\infty$. Rather, we will focus on an important class of univalent morphisms, namely the {\it univalent (co)Cartesian fibrations}. For that we need to review the {\it universal Cartesian fibration}.

Recall, by \fibref{item:grothendieck}, that for a given large $\infty$-category $\C$ there is a natural equivalence between Cartesian fibrations and (very large) $\infty$-category valued functors 
\begin{equation} \label{eq:cart}
(\Cart_{/\C^{op}})^\simeq \simeq \Map(\C,\widehat{\cat}_\infty).
\end{equation}
Fix a cardinal $\kappa$. Then the equivalence \ref{eq:cart} restricts to an equivalence of large $\infty$-categories
\begin{equation} \label{eq:smcart}
	(\Cart^{\kappa}_{/\C^{op}})^\simeq \simeq \Map(\C,\cat^\kappa_\infty),
\end{equation}
where $\cat^\kappa_\infty$ denotes the large $\infty$-category of $\kappa$-small $\infty$-categories and $\Cart^{\kappa}$ denotes {\it Cartesian fibrations with $\kappa$-small fibers}. By the Yoneda lemma (\fibref{item:rep right fib}) this equivalence must be given by pulling back along a  Cartesian fibration over $(\cat^\kappa_\infty)^{op}$ that under the equivalence \ref{eq:smcart} corresponds to the identity $\cat^\kappa_\infty \to \cat^\kappa_\infty$. Given the universal property, we call this Cartesian fibration, the {\it universal $\kappa$-small Cartesian fibration} and denote it $p_{\Cart^\kappa}$. This Cartesian fibration can in fact be characterized in more detail, as has been done in \cite[Subsection 3.3.2]{lurie2009htt}, \cite[Example 4.9]{stenzel2020comprehension} or \cite[Subsection 4.2]{rasekh2017cartesian}. Here we want to focus on how we can use this universal Cartesian fibration to characterize univalent Cartesian fibrations.  
 
First of all, denote by $\sO^{\Cart}_{\widehat{\cat}_\infty} \hookrightarrow \Oall_{\widehat{\cat}_\infty}$ the full sub-right fibration of Cartesian fibrations in $\widehat{\cat}_\infty$ and note it corresponds (via \fibref{item:grothendieck}) to the functor $(\Cart_{/-})^\simeq: (\widehat{\cat}_\infty)^{op} \to \widehat{\s}$. As always, this cannot possibly be representable because of size issues, hence, let $\sO^{\Cart^\kappa}_{\widehat{\cat}_\infty}$ be a further restriction to $\kappa$-small Cartesian fibrations. The equivalence  \ref{eq:smcart} precisely states that the right fibration $\sO^{\Cart^\kappa}_{\widehat{\cat}_\infty}$ is representable (via $\cat^\kappa_\infty$). This means we have inclusions
\begin{equation} \label{eq:pcart univalent}
	\Map_{\widehat{\cat}_\infty}(-,\cat_\infty^\kappa) \simeq \sO^{\Cart^\kappa}_{\widehat{\cat}_\infty} \hookrightarrow \sO^{\Cart}_{\widehat{\cat}_\infty} \hookrightarrow \Oall_{\widehat{\cat}_\infty},
\end{equation}
which, by \cref{prop:univ final}, proves that the universal $\kappa$-small Cartesian fibration is univalent. 

\begin{remone}
	This result has also been proven via the original approach \cite{kapulkinlumsdaine2012kanunivalent} by Cisinski and Nguyen in the strict setting using the Cartesian model structure on marked simplicial sets \cite{cisinskynguyen2020hottest}.
\end{remone} 

Using this observation we can give a general characterization of univalent Cartesian fibrations, analogous to \cref{prop:univ n topos}.

\begin{propone} \label{prop:univalent coCart}
	A $\kappa$-small Cartesian fibration $p: \D \to \C$ is univalent if and only if the corresponding map $\C^{op} \to \cat^\kappa_\infty$ (given via \ref{eq:smcart}) is a mono map of $\infty$-categories.
\end{propone}

\begin{proof}
	This follows directly from the fact that $p_{\Cart}$ is univalent combined with \cref{prop:univ if mono}.
\end{proof}

\begin{remone} \label{rem:ff vs mono}
	It is important to remember that mono maps of $\infty$-categories are {\bf not} the fully faithful functors. Rather they are the functors that are characterized by a subspace of objects and for any two objects, a subspace of the mapping space. For example, the inclusion from the maximal sub-$\infty$-groupoid is always mono, but very rarely fully faithful.
\end{remone}

In fact we can also externalize this condition and generalize \ref{eq:univ ff}. Let $p: \D \to \C$ be an arbitrary Cartesian fibration. Then it comes with a corresponding functor $\C^{op} \to \widehat{\cat}_\infty$ (given via \ref{eq:cart}). Choosing a large enough $\kappa$, this functor will factor  $\C^{op} \to \cat^\kappa_\infty \hookrightarrow \widehat{\cat}_\infty$. Now, by \cref{prop:univalent coCart}, $p$ is univalent if and only if the map $\C^{op} \to \cat^\kappa_\infty$ is mono, which is equivalent to the composition $\C^{op} \to \widehat{\cat}_\infty$ being mono as $\cat^\kappa_\infty$ is a full sub-$\infty$-category of $\widehat{\cat}_\infty$. Hence, we get following characterization of univalence relying on the very large $\infty$-category $\widehat{\cat}_\infty$.

\begin{corone} \label{cor:univ cart}
 A Cartesian fibration $p:\D \to \C$ is univalent if and only if the corresponding functor $\C^{op} \to \widehat{\cat}_\infty$ (given via \ref{eq:cart}) is mono. 
\end{corone}
 
 In other words for a given large $\infty$-category $\C$, there is an equivalence of (very large) posets
 \begin{equation} \label{eq:univ mono cat}
 	(\Univ^{\Cart})_{/\C} \simeq \Fun^{mono}(\C^{op},\widehat{\cat}_\infty),
 \end{equation}
 where $(\Univ^{\Cart})_{/\C}$ denotes the poset of univalent Cartesian fibrations over $\C$ and $\Fun^{mono}$ are mono functors. Notice the similarity to \ref{eq:univ ff} despite the fact that $\cat_\infty$ is not an $\infty$-topos or even locally Cartesian closed. 

It is worth pointing out that there are similar results to \cref{prop:univalent coCart} and \cref{cor:univ cart} for coCartesian fibrations.
Using these results (and the analogous ones for coCartesian fibrations) we can construct many interesting univalent morphisms in $\widehat{\cat}_\infty$. In particular, the {\it universal  right fibration} (studied extensively in \cite[Lemma 2.2.4]{kazhdanvarshvsky2014yoneda} \cite[Corollary 5.2.8]{cisinski2019highercategories}) is obtained via the inclusion of $\infty$-categories $\widehat{\s} \to \widehat{\cat}_\infty$ and so, by \cref{cor:univ cart}, must be univalent. Similarly, the universal $n$-truncated Cartesian fibration is obtained via the inclusion of the full sub-category of $n$-truncated $\infty$-categories $(\widehat{\cat}_\infty)_{\leq n} \to \widehat{\cat}_\infty$ and so must be univalent as well.

Until now we have used the external approach (via the target fibration) to study univalence in $\cat_\infty$. We want to understand how it relates to the internal version, using \cref{the:main theorem}. The first problem we encounter is that $\cat_\infty$ is not locally Cartesian closed. For example pulling back along the inclusion of posets $\{0 \leq 2\} \hookrightarrow \{0 \leq 1 \leq 2\}$ does not preserve pushouts and so cannot have a right adjoint (implying that even the sub-category $\cat$ is not locally Cartesian closed). However, we can still use the internal characterization of univalence for all functors for which pulling back admits a right adjoint. Such functors are known as {\it exponentiable fibration} \cite{ayalafrancis2020fibrations} (also called {\it Condouch{\'e} fibrations} in the classical setting \cite{conduche1972fibrations}). Fortunately, exponentiable fibrations do include (co)Cartesian fibrations and so we can in fact give internal characterization of univalence. In particular, we can combine \cref{the:main theorem} and \ref{eq:pcart univalent} to get.

\begin{corone}
 Let $p_{\Cart}$ be the universal coCartesian fibration. Then the Segal object $\n(p_{\Cart})$ is complete.
\end{corone}

Let us unpack the result. $\n(p_{\Cart})$ is a Segal object in $\cat_\infty$, which gives us a double $\infty$-category as defined in \cite{haugseng2017highermorita}. This double $\infty$-category is characterized by objects being $\infty$-categories, horizontal and vertical morphisms being functors of $\infty$-categories, and squares being commutative squares of such functors. The completeness of the double $\infty$-category corresponds to the observation that vertical morphism and horizontal morphisms coincide and hence a vertical morphism is an equivalence if and only if the horizontal morphism is an equivalence. 

From another perspective the comparison between \ref{eq:univ ff} and \cref{prop:univalent coCart} seems unsatisfying. Unlike a space an $\infty$-category has ``directed data", meaning morphism of our category and those should be considered in our univalence condition. For example, let $\C$ be an $\infty$-category and denote the full sub-$\infty$-category of $\widehat{\cat}_\infty$ with object $\C$ by $B\End(\C)$ and its sub-$\infty$-groupoid by $B\Aut(\C)$. Then the functors
$$B\Aut(\C) \hookrightarrow B\End(\C) \hookrightarrow \widehat{\cat}_\infty$$
are all mono and hence, by \cref{cor:univ cart}, the pullback of $p_{\Cart}$ along these monos is univalent as well. However, this does not match with \ref{eq:univ ff}. Ideally we would like pulling back to $B\End(\C)$ to be univalent, but pulling back to $B\Aut(\C)$ to not be as only the first is actually fully faithful in $\cat_\infty$ (\cref{rem:ff vs mono}).

That would require us to have a notion of ``directed univalence", which we cannot define in an $\infty$-category. For that we would have to use the fact that $\cat_\infty$ is in fact an $(\infty,2)$-category and develop a $2$-categorical notion of univalence, which is expected to relate to $2$-topos theory \cite{weber2007twotopos,riehl2018jmm,riehl2019ct}. More importantly, it could help us develop a theory of $(\infty,2)$-toposes. Concretely a key concept in the development of $(\infty,1)$-topos theory is the notion of {\it descent} \cite{rezk2010toposes}, also known as the {\it Giraud axioms} \cite{lurie2009htt}, which should be thought of as an $\infty$-categorical sheaf condition. It is this descent condition that guarantees the existence of object classifiers, which we can use to characterize univalent morphisms as has been done in \cite[Corollary 3.11]{gepnerkock2017univalence} and \cref{prop:univ n topos}.

On the other hand it is not yet understood what the $(\infty,2)$-categorical analogue of the sheaf condition should be and developing a $(\infty,2)$-categorical notion of univalence could serve as a guide towards properly developing such a sheaf condition.

{\bf Pointed $\infty$-Categories:} As we saw in \ref{eq:infty topos} a Grothendieck $\infty$-topos is a sheaf of spaces and hence generalizes spaces by introducing ``general elements". However, we can generalize spaces in other ways, unrelated to topos theory. Key examples of such generalizations are pointed spaces, $E_\infty$-monoids and spectra. We want to make some steps towards analyzing univalent morphisms in these categories.

Let us review the relevant concepts first. An $\infty$-category is called {\it pointed} if it has an initial and final object and they are equivalent. Moreover, we say an $\infty$-category is stable if it has finite limits and colimits and they commute. There are various other $\infty$-categories that fall between pointed and stable $\infty$-categories and are studied in detail in \cite{lurie2017ha,gepnergrothnikolaus2015infiniteloopspacemachine}. In particular all these types of $\infty$-categories are pointed and so we will focus on this case here. 

Assuming $\C$ is an $\infty$-category with final object $*$, the under-category $\C_{*/}$, which we simply denote $\C_*$, is by definition a pointed $\infty$-category. In fact this construction gives us an adjunction between $\infty$-categories and pointed $\infty$-categories
\begin{center}
	\begin{tikzcd}
		\cat_\infty \arrow[r, shift left=1.8, "(-)_*", "\bot"'] \arrow[r, shift right=1.8, hookleftarrow] & \cat_\infty^{pt} 
	\end{tikzcd}
\end{center}
where the left adjoint is simply the inclusion of the full subcategory of pointed $\infty$-categories. Hence, we want to focus on determining the univalent morphisms in $\C_*$.

The first problem we encounter is that $\C_*$ is not Cartesian closed. Indeed, let us denote the initial object in $\C_*$ by $0_\C$ and the final object by $1_\C$ and assume we have an internal mapping object $\C_*(c,d)$. Then, for objects $c,d$ we have equivalences
$$\Map_{\C_*}(c,d) \simeq \Map_{\C_*}(1_\C \times c,d) \simeq \Map_{\C_*}(1_\C,\C_*(c,d)) \simeq \Map_{\C_*}(0_\C,\C_*(c,d)) \simeq *,$$
which implies that $\Map_{\C_*}(c,d)$ must be trivial.
Hence, we have to exclusively rely on the original definition of univalence given in \cref{def:univalence}.

Let us use this observation to classify some interesting univalent morphisms.
Fix a morphism $1_\C \to E \xrightarrow{ \ p \ } B$ in the $\infty$-category $\C_*$, such that $1_\C \to E$ is $(-1)$-truncated. Then for every morphism $1_\C \to Y \xrightarrow{ \ f \ } X$, we have following pullback diagram 
\begin{equation} \label{eq:pointed}
	\begin{tikzcd}
		\Map_{\Oall_{\C_*}}(1_\C \to Y \xrightarrow{ \ f \ } X, 1_\C \to E \xrightarrow{ \ p \ } B)\arrow[r] \arrow[d, hookrightarrow] & 
		\Map_{\C_*}(1_\C \to Y , 1_\C \to E) \arrow[r] \arrow[d] & * \arrow[d,"1_\C", hookrightarrow] \\
		\Map_{\Oall_{\C}}(Y \xrightarrow{ \ f \ } X,E \xrightarrow{ \ p \ } B) \arrow[r] & \Map_{\C}(Y ,E) \arrow[r] & \Map_\C(1_\C,E) 
	\end{tikzcd}.
\end{equation}
Here we are using the fact that the projection $\C_* \to \C$ reflects pullback squares \cite[Proposition 1.2.13.8]{lurie2009htt}.
The map $* \to \Map_\C(1_\C,E) $ is a $(-1)$-truncated map of spaces (\fibref{item:over cat trunc}). As the map on the right is $(-1)$-truncated, the left hand map is $(-1)$-truncated as well. Hence, $1_\C \to E \xrightarrow{ \ p \ } B$ is univalent in $\C_*$ only if $p:E \to B$ is univalent in $\C$, meaning the study of univalent maps can be reduced to the unpointed case. For example, for any space $X$, using \ref{eq:pointed} and \ref{eq:univ spaces} we can deduce that the map of pointed spaces $* \to E(\Eq(X)) \to B(\Eq(X))$ is univalent. Notice this example also helps us provide a counter-example for the other direction. The map $\emptyset \to \emptyset$ is univalent in $\s$ but cannot be recovered from a univalent morphism in $\s_*$ as there are no morphisms $ * \to \emptyset$. 

We can use this result to give a broader result about univalence in pointed categories coming from Grothendieck $1$-toposes.

\begin{corone}
	Let $\G$ be a Grothendieck $1$-topos. Then a morphism $1_\G \to E \to B$ in $\G_*$ is univalent only if the map $B \to \U$ classifying $p:E \to B$ is mono in the enveloping Grothendieck $\infty$-topos $\hat{\G}$.
\end{corone}

\begin{proof}
	As $\G$ is a $1$-category, any morphism $1_\G \to E$ is $(-1)$-truncated. Hence, following \ref{eq:pointed}, $1_\C \to E \to B$ is univalent only if $E \to B$ is univalent in $\G$. By \cref{prop:univ n topos} this is equivalent to the classifying morphism being mono in the enveloping Grothendieck $\infty$-topos.
\end{proof}

\bibliographystyle{alpha}
\bibliography{main}

\end{document}